\documentclass[10pt]{amsart}
\usepackage{amsmath}
\usepackage{amsthm}
\usepackage{amscd}
\usepackage{amssymb}
\usepackage{amsfonts}
\usepackage{graphics}
\usepackage[dvipsnames]{xcolor}
\usepackage[latin1]{inputenc}
\usepackage[T1]{fontenc}
\usepackage{soul}
\usepackage{cancel}
\usepackage{geometry}

\date{}

\newlength{\defbaselineskip}
\long\def\salta#1{\relax}

\theoremstyle{plain}
\newtheorem{theorem}{Theorem}[section]
\newtheorem{proposition}[theorem]{Proposition}
\newtheorem{lemma}[theorem]{Lemma}

\newtheorem{corollary}[theorem]{Corollary}

\theoremstyle{definition}
\newtheorem{definition}[theorem]{Definition}

\newtheorem{remark}[theorem]{Remark}

\theoremstyle{remark}

\newcommand{\na}{\mathbb{N}}

\newcommand{\re}{\mathbb{R}}

\definecolor{jose}{rgb}{0.3,0.4,0.8}

\def\dys{\displaystyle}

\def\t1p0{T^{1,p}_{0}(\Omega)}

\def\m2{M^{\frac{N(p-1)}{N-1}}(\Omega)}

\def\w-1p'{W^{-1,p'}(\Omega)}
\def\pw-1p'u{L^{p'}(0,1;W^{-1,p'}(\Omega))}

\def\dys{\displaystyle}

\def\lp'n{(L^{p'}(\Omega))^{N}}

\def\supp{\text{\rm{supp}}}

\numberwithin{equation}{section}

\title[A break point for the multiplicity of solutions]{A singularity as a break point for the multiplicity of solutions to quasilinear elliptic problems}

\keywords{Nonlinear elliptic equations, Singular gradient terms, Multiplicity of solutions, Uniqueness of solution.
\\
\indent Research supported by MINECO-FEDER grant MTM2015-68210-P, Junta de Andaluc\'ia FQM-116 and Programa de Contratos Predoctorales del Plan Propio de la Universidad de Granada.
}

\begin{document}

\maketitle

\begin{center}
{\small SALVADOR L\'OPEZ-MART\'INEZ\\[0.4 cm] Departamento de An\'alisis Matem\'atico\\ Universidad de Granada\\ Facultad de Ciencias, Avenida Fuentenueva s/n, 18071, Granada, Spain\\[0.4 cm]
\textsl{Email adress: salvadorlopez@ugr.es}}
\end{center}

\hrule

\begin{abstract}
In this paper we deal with the elliptic problem
\begin{equation*}
\begin{cases}
\dys-\Delta u=\lambda u+\mu(x)\frac{|\nabla u|^q}{u^\alpha}+f(x)\quad &\text{ in }\Omega,
\\
u>0 \quad &\text{ in }\Omega,
\\
u=0\quad &\text{ on }\partial\Omega,
\end{cases}
\end{equation*}
where $\Omega\subset\mathbb{R}^N$ is a bounded smooth domain, $0\lneqq\mu\in L^\infty(\Omega)$, $0\lneqq f\in L^{p_0}(\Omega)$ for some $p_0>\frac{N}{2}$, $1<q<2$, $\alpha\in [0,1]$ and $\lambda\in\re$. We establish existence and multiplicity results for $\lambda>0$ and $\alpha<q-1$, including the non-singular case $\alpha=0$. In contrast, we also derive existence and uniqueness results for $\lambda>0$ and $q-1<\alpha\leq 1$. We thus complement the results in \cite{CLM, CLLM}, which are concerned with $\alpha=q-1$, and show that the value $\alpha=q-1$ plays the role of a break point for the multiplicity/uniqueness of solution.
\end{abstract}

\hrule

\section{Introduction}\label{sec:introduction}

In this paper we deal with the following boundary value problem:
\begin{equation}\label{problem}
\begin{cases}
\dys-\Delta u=\lambda u+\mu(x)\frac{|\nabla u|^q}{u^\alpha}+f(x)\quad &\text{ in }\Omega,
\\
u>0 \quad &\text{ in }\Omega,
\\
u=0\quad &\text{ on }\partial\Omega.
\end{cases}
\tag{$P_\lambda$}
\end{equation}
Here, $\Omega$ is a bounded domain of $\re^N$ ($N\geq 3$) with  boundary $\partial\Omega$ smooth enough, $0~\lneqq~\mu~\in~ L^\infty(\Omega)$, $0\lneqq f\in L^{p_0}(\Omega)$ for some $p_0>\frac{N}{2}$, $1<q<2$, $0\leq\alpha\leq 1$ and $\lambda\in\re$. A solution to \eqref{problem} is a function $0<u\in H_0^1(\Omega)\cap L^\infty(\Omega)$ which satisfies the equation in \eqref{problem} in the usual weak sense (we will be more precise about the concept of solution in Definition \ref{solution} below). Observe that, if $\alpha>0$, then the lower order term presents a singularity as $u$ approaches zero, i.e., as $x$ approaches $\partial\Omega$. Our goal is to study the existence, nonexistence, uniqueness and multiplicity of solutions to \eqref{problem}, specially for $\lambda>0$. 

\smallskip

The first motivation for dealing with this problem comes from the non-singular case $\alpha=0$, i.e.,
\begin{equation}\label{problemns}
\begin{cases}
\dys-\Delta u=\lambda u+\mu(x)|\nabla u|^q+f(x)\quad &\text{ in }\Omega,
\\
u>0 \quad &\text{ in }\Omega,
\\
u=0\quad &\text{ on }\partial\Omega.
\end{cases}
\tag{$R_\lambda$}
\end{equation}
It is well-known from classical results (see \cite{BMP1,BMP2}) that problem \eqref{problemns} admits at least one solution for all $\lambda<0$. Concerning the uniqueness of solution, it was first dealt with in \cite{BarlesMurat}, and their results have been improved in several directions since then (see \cite{ACJT2} and references therein). In particular, it has been recently proved in \cite{ACJT2} that uniqueness holds for all $\lambda\leq 0$. However, the existence of solution for $\lambda=0$ is not always guaranteed. Roughly speaking, if $\|f\|_{L^{p_0}(\Omega)}$ is small enough, then there exists a unique solution to $(R_0)$, as it is shown for instance in \cite{FPR} (see also \cite{GMP} and references therein). Conversely, it is proved in \cite{AP} (see also \cite{HMV}) that, if $f$ is large in some sense, there exists no solution to $(R_0)$; in consequence, $\lambda=0$ is a bifurcation point from infinity. Concerning this last case, a very precise description of the blow-up of the solutions at $\lambda=0$, and also a necessary and sufficient condition for the existence of solution to $(R_0)$ in terms of the corresponding ergodic problem, are given in \cite{P} under slightly stronger hypotheses on $f$ and $\mu$. 

\smallskip

The scenario in which $(R_0)$ has a solution is not so well understood, and has risen interest in the recent years. In this case one expects to find solutions to \eqref{problemns} for small $\lambda>0$ by a continuation argument. However, the uniqueness and multiplicity problems are harder to deal with for $\lambda>0$, and very few results are known in this direction. In fact, up to our knowledge, the literature contains results concerning only the quadratic case $q=2$. In this regard, the first advances can be found in \cite{JS} for $\mu>0$ constant. Shortly after that, some improvements appeared in \cite{JQ}, where $\lambda=\lambda(x)$ is allowed to change sign but $\mu$ is still constant. These two works employ variational techniques. Going further, topological degree and bifurcation are used in \cite{ACJT1} to handle problem \eqref{problemns} with $\lambda>0$ and $\mu\in L^\infty(\Omega)$ such that $\mu_1\leq\mu\leq\mu_2$ for some constants $\mu_2>\mu_1>0$. We also quote \cite{Souplet}, where functions $0\lneqq\mu\in L^\infty(\Omega)$ vanishing on $\partial\Omega$, and even with compact support, are permitted at the expense of imposing $N\leq 3$ (the cases $N=4, 5$ are also handled provided $\lambda=\lambda(x)$ satisfies extra hypotheses). Very recently, a similar problem to \eqref{problemns} with the $p$ -Laplacian as principal operator has been considered in \cite{CF}, while sign-changing coefficients (including $\mu$) are allowed in \cite{CFJ}. 

\smallskip

In all these works, the authors prove that, if there is a solution to $(R_0)$, then problem \eqref{problemns} admits at least two different solutions for all $\lambda>0$ small enough, and it was first shown in \cite{ACJT1} that the branch of positive solutions bifurcates from infinity to the right of the axis $\lambda=0$ (see \cite{CJ} for a more complete picture when different sign conditions on $f$ are imposed). We stress again that all the mentioned papers have in common the assumption $q=2$. Indeed, the techniques employed for $q=2$ usually involve exponential test functions which somehow remove the dependence on the gradient in the equation. For instance, this idea allows the authors of \cite{JS} to study the problem variationally, while in \cite{ACJT1} it is essential in order to find a priori estimates for $\lambda>0$. However, this idea fails for $1<q<2$ as the gradient term can not be removed when one looks for a priori estimates satisfied by supersolutions to \eqref{problemns}. Up to our knowledge, the multiplicity or uniqueness of solutions for $\lambda>0$ is an open problem if $1<q<2$. 

\smallskip

Turning back to \eqref{problem}, another motivation for studying this problem comes from the very recent paper \cite{CLLM}. In Remark 6.1 of that paper the authors observe that, if $q=2$ and $0<\alpha<q-1$, the techniques in \cite{ACJT1} can be adapted to derive again a multiplicity result for $\lambda>0$. Hence, roughly speaking, mild singularities at zero do not alter the behavior of the solutions, as far as the multiplicity for $\lambda>0$ is concerned. Nonetheless, the main result in that paper shows that multiplicity fails for $\alpha=q-1$ (see \cite{AM} for $q=2$ and $\mu$ constant). To be precise, the authors prove under natural hypotheses on $\mu$ and $f$ that, if $\alpha=q-1$, there exists $\lambda^*\in (0,\lambda_1]$ (where $\lambda_1=\inf_{v\in H_0^1(\Omega)\setminus\{0\}}\int_\Omega|\nabla v|^2/\int_\Omega v^2$) such that problem \eqref{problem} has a solution if and only if $\lambda<\lambda^*$, and in this case, the solution is unique (see also \cite{CLM} for a similar existence result when $f$ and $u$ may change sign). In particular, one has existence and uniqueness for $\lambda>0$ small. Since this result is true for $1<q\leq 2$, it is natural to wonder whether $\alpha=q-1$ is a break point for the multiplicity of solutions not only in the case $q=2$, but also for $1<q<2$. 

\smallskip

In the present work we contribute to these topics by proving that, if there is a solution to $(P_0)$, then there are at least two different solutions to \eqref{problem} for all $\lambda>0$ small enough provided $q$ and $\alpha$ satisfy certain relations involving also the dimension $N$. We prove also that the branch of positive solutions bifurcates from infinity to the right of the axis $\lambda=0$. 

\smallskip

To be more precise, we consider the following set of hypotheses:
\begin{equation}\tag{H1}\label{H1}
\left\{
\begin{array}{l}
\Omega\subset\re^N\text{ is a bounded domain of class }\mathcal{C}^2,
\\
\mu\in L^\infty(\Omega)\text{ satisfies that } \mu\geq\mu_0 \text{ in } \Omega\text{ for some constant }\mu_0>0,
\\
0\lneqq f\in L^{p_0}(\Omega)\text{ for some }p_0>\frac{N}{2},
\\
q\in (1,2), 
\\
\alpha\in [0,q-1).
\end{array}
\right.
\end{equation}
Observe that $\mu$ is bounded away from zero but not necessarily constant. We introduce here the main result of this paper:

\begin{theorem}\label{maintheorem}
Assume that \eqref{H1} holds and that $(P_0)$ admits a solution $u_0$. If $q>\frac{N}{N-1}$, suppose also that 
\begin{equation}\label{mrcondition}
\frac{q-1-\alpha}{q-2\alpha}\leq \frac{q-\alpha}{N-q+1}.
\end{equation}
Furthermore, if $q\geq 1+\frac{2}{N}$, assume additionally that 
\begin{equation}\label{qlargecondition}
\alpha\leq \frac{q(N+4)-2(N+1)}{N+2}.
\end{equation}
Then, there exists $\bar{\lambda}\in (0,\lambda_1)$ such that problem \eqref{problem} admits at least two different solutions for all $\lambda\in (0,\bar{\lambda}]$. Moreover, zero is the unique bifurcation point from infinity to problem \eqref{problem}.
\end{theorem}

Even though this result deals only with the range $\lambda>0$, in order to make a more complete picture we will gather and prove in Section \ref{sec:superlinear} some existence, nonexistence and uniqueness results about problem \eqref{problem} for $\lambda\leq 0$. We stress that the uniqueness result for $\lambda\leq 0$, apart from being new in the literature, shows that $\lambda=0$ is a critical point beyond which the nature of the problem changes drastically, as in the well-known case $q=2$ and $\alpha=0$.

\smallskip

Concerning the proof of Theorem \ref{maintheorem}, the idea is to derive a priori estimates of the solutions to \eqref{problem} for all $\lambda>\lambda_0$ which are independent of $\lambda>0$. This idea first appeared in \cite{ACJT1} for $q=2$ and $\alpha=0$, but the approach for deriving the estimates does not work in our framework. For our purposes, it is more convenient to use the arguments developed in \cite{Souplet}, which allow us to find $L^p$ estimates of supersolutions. After that, we establish a bootstrap argument, which works thanks to some results in \cite{GMP}, that yields an $L^\infty$ estimate. Actually, these results are valid only in the nonsingular case $\alpha=0$, so we will extend some parts of them to our singular framework.

\smallskip

Hypotheses \eqref{mrcondition} and \eqref{qlargecondition} in Theorem \ref{maintheorem} deserve some comments. They appear in the proof as a result of the combination of the mentioned techniques from \cite{Souplet} and the bootstrap from \cite{GMP}. However, we presume that these are technical assumptions forced by the tools we employed, so the theorem might admit some improvements. In order to clarify the meaning of these two conditions, we derive some corollaries below in which simpler conditions assuring \eqref{mrcondition} are imposed. For instance, if we consider the sequence
\begin{equation}\label{Qn}
Q_n=\left\{
\begin{array}{ll}
\displaystyle 2&\forall n\leq 4,
\\
\displaystyle\frac{n+2-\sqrt{n^2-4n-4}}{4}&\forall n\geq 5,
\end{array}
\right.
\end{equation} 
then $q\in (1,Q_N]\setminus\{2\}$ implies \eqref{mrcondition}, with no extra hypotheses on $\alpha$ apart from $0\leq\alpha<q-1$ (see Corollary \ref{corollary1}). Observe that  
$Q_n>1$ but $\lim_{n\to\infty}Q_n=1$. This means that, if $N$ is large, then $q$ has to be chosen close to $1$. However, one would expect a multiplicity result for any $q\in (1,2)$ and any $N$. This still remains as an open problem. In any case, Corollary \ref{corollary1} represents a remarkable advance, in particular, about the nonsingular problem \eqref{problemns}. Changing the point of view, we give in Corollaries \ref{corollary2} and \ref{corollary3} below conditions on $\alpha$ and $N$ that are sufficient for applying Theorem \ref{maintheorem} even for $q$ close to $2$. 

\smallskip

With the aim of having a deeper insight into problem \eqref{problem}, we also consider in this work the case $q-1<\alpha\leq 1$. In contrast to the previous situation ($0\leq \alpha<q-1$), we will prove that existence and \emph{uniqueness} hold for $\lambda>0$ small enough. For this purpose, we will need the following assumption on $\Omega$:
\begin{equation}\tag{A}\label{condA}
\left\{
\begin{array}{l}
\text{There exist }r_0, \theta_0 > 0 \text{ such that, if } x\in \partial \Omega\text{  and }0 < r < r_0, \text{ then}
\\
|\Omega_r| \leq (1 - \theta_0)|B_r(x)|\text{ for every connected component }\Omega_r\text{ of }\Omega\cap B_r(x).
\end{array}
\right.
\end{equation}
Note that, if $\partial\Omega$ is Lipschitz, then $\Omega$ satisfies \eqref{condA} (see \cite{ACJT2}), so this represents only a mild restriction. The precise hypotheses that we need are gathered here:
\begin{equation}\tag{H2}\label{H2}
\left\{
\begin{array}{l}
\Omega\subset\re^N\text{ is a bounded domain satisfying condition }\eqref{condA},
\\
0\lneqq\mu\in L^\infty(\Omega),
\\
0\lneqq f\in L^{p_0}(\Omega)\text{ for }p_0>\frac{N}{2}, 
\\
q\in (1,2),
\\
q-1<\alpha\leq 1.
\end{array}
\right.
\end{equation}
We emphasize that $\mu$ is allowed to vanish in subsets of $\Omega$ with nonzero measure. 

\smallskip

The statement of the main result in the $q-1<\alpha\leq 1$ case is the following:
\begin{theorem}\label{sublinearthm}
Assume that \eqref{H2} holds. Then there exists a solution to \eqref{problem} for all $\lambda<\lambda_1$, and there exists no solution to \eqref{problem} for all $\lambda\geq\lambda_1$. Moreover, the solution is unique for all $\lambda\leq 0$ and, if $f$ satisfies that
\[\forall\omega\subset\subset\Omega\quad\exists c_\omega>0:\quad f\geq c_\omega\quad\text{ in }\omega,\]
then the solution is unique for all $\lambda<\lambda_1$. Finally, $\lambda_1$ is the unique bifurcation point from infinity to problem \eqref{problem}.
\end{theorem}

Even though we are specially interested in the uniqueness part, the existence statement in Theorem \ref{sublinearthm} deserves also attention. Observe that one has existence of solution if and only if $\lambda<\lambda_1$. This suggests that the nonlinear term does not play an essential role in this case, since the situation is analogous to the linear problem ($\mu\equiv 0$). Recall that this is not the case when $\alpha=q-1$, for which one has existence if and only if $\lambda<\lambda^*$, where $\lambda^*<\lambda_1$ provided $\mu>0$ (see \cite[Remark 6.3]{CLLM}). 

\smallskip

The proof of the existence of solution in Theorem  \ref{sublinearthm} is performed by passing to the limit in certain family of approximate nonsingular problems. We will derive H\"older continuous a priori estimates on the solutions to such a family, which will allow us to pass to the limit. For proving such estimates, the assumption $\alpha\leq 1$ is essential (see Remark \ref{lambda1remark1} below). Moreover, the continuity of the solutions is also essential to prove their uniqueness. Indeed, we state and prove in Section \ref{sec:comparison} two comparison principles valid for continuous lower and upper solutions to singular equations. As far as we know, these two results are new, and they are interesting by themselves as only few uniqueness results for singular equations are known (see \cite{ACM,AM,AS,CL,CLLM}). We follow in their proofs the arguments in \cite{ACJT2} and \cite{CLLM}.

\smallskip

As a summary, our results contribute to the theory of equations with subquadatic growth in the gradient, extending what it is known about the multiplicity of solutions in the quadratic case. On the other hand, they can be seen as a link between the singular and nonsingular theory, in the sense that they show that the presence or not of a singularity is determining only if it is strong enough. Finally, new  existence and uniqueness results are given for strong singularities, where the uniqueness part is specially remarkable.

\smallskip

We organize the paper as follows: in Section \ref{sec:comparison} we deal with the mentioned comparison principles; we devote Section \ref{sec:superlinear} to prove Theorem \ref{maintheorem} as well as some auxiliary results and some consequences of the mentioned theorem; Section \ref{sec:sublinear} contains the proof of Theorem \ref{sublinearthm}, and Section \ref{sec:appendix} is an appendix where we prove a continuation result needed in the proof of Theorem~\ref{maintheorem}.

\subsection*{Acknowledgments}

The author wants to thank warmly T. Leonori and J. Carmona for their helpful contributions to this work.

\subsection*{Notation}

\begin{itemize}
\item For every $x\in\mathbb{R}^N$, the distance from $x$ to $\partial\Omega$ will be denoted as $\delta(x)$. Furthermore, for $p\geq 1$ we will denote as $L^{p,\delta}(\Omega)$ the space of functions $u:\Omega\to\mathbb{R}$ such that \[\|u\|_{L^{p,\delta}(\Omega)}:=\left(\int_\Omega |u(x)|^p\delta(x)dx\right)^\frac{1}{p}<+\infty,\] identifying functions equal up to a set of zero measure.
\item For $p\geq 1$, we will denote the usual Marcinkiewicz space as $\mathcal{M}^p(\Omega)$, i.e., the space of functions $u:\Omega\to\mathbb{R}$ for which there exists $c>0$ such that $|\{|u|>k\}|k^p\leq c$ for all $k>0$. In this case, we denote  \[\|u\|_{\mathcal{M}^p(\Omega)}:=\left(\inf\{c>0:|\{|u|>k\}|k^p\leq c\text{ for all }k>0\right)^\frac{1}{p}.\]
\item For $k\geq 0$, the usual truncation functions will be written as $T_k(s)=\max\{-k,\min\{s,k\}\}$ and $G_k(s)=s-T_k(s)$ for all $s\in\re$.
\item The principal eigenvalue of the $-\Delta$ operator under zero Dirichlet boundary conditions will be denoted as $\lambda_1$, and $\varphi_1$ will denote the corresponding eigenfunction with $\|\varphi_1\|_{L^\infty(\Omega)}=1$.
\end{itemize}

\section{Comparison principles}\label{sec:comparison}

We start with a comparison principle valid for singular equations. The proof basically follows the steps of a similar result in \cite{ACJT2}. However, up to our knowledge this is the first time that a comparison result has been proved including a general positive singular lower order term on the right hand side of the equation (see the comparison results in \cite{CLLM}, where a specific $1$-homogeneous singular term is considered).

\begin{theorem}\label{comprinc1}
Let $1<q\leq 2$, $\lambda\leq 0$, $h\in L^1_{\mbox{\tiny loc}}(\Omega)$ and $g:\Omega\times (0,+\infty)\to [0,+\infty)$ satisfying
\begin{align*}
s&\mapsto g(x,s)\,\,\text{ is nonincreasing for a.e. }x\in\Omega,
\\
x&\mapsto g(x,s)\,\,\text{ is locally essentially bounded for all }s>0.
\end{align*}
Let $u, v\in C(\Omega)\cap W^{1,N}_{\mbox{\tiny loc}}(\Omega)$, with $u,v>0$ in $\Omega$, be such that
\begin{align}
\label{compsub}
&\int_\Omega \nabla u\nabla\phi\leq \lambda \int_\Omega u\phi+\int_\Omega g(x,u)|\nabla u|^q\phi+\int_\Omega h(x)\phi\quad\text{ and }
\\
\label{compsuper}
&\int_\Omega \nabla v\nabla\phi\geq \lambda \int_\Omega v\phi+\int_\Omega g(x,v)|\nabla v|^q\phi+\int_\Omega h(x)\phi
\end{align}
for every $0\leq \phi\in H_0^1(\Omega)\cap L^\infty(\Omega)$ with compact support. 
Suppose also that the following boundary condition holds:
\begin{equation}\label{boundarycond}
\limsup_{x\to x_0}(u(x)-v(x))\leq 0 \quad \forall x_0\in\partial\Omega.
\end{equation}
Then, $u\leq v$ in $\Omega$. 
\end{theorem}

\begin{remark}
Theorem \ref{comprinc1} is valid for a wide class of lower order terms. For instance, the model example is 
\[g(x,s)=\frac{\mu(x)}{s^\alpha}\quad\text{a.e. }x\in\Omega,\,\,\forall s>0,\]for any $\alpha>0$ and $0\leq\mu\in L_{\mbox{\tiny loc}}^\infty(\Omega)$. In particular, the growth of the singularity is irrelevant in the proof. Nonetheless, the comparison principle does not work for $\lambda>0$. Indeed, as we pointed out in the Introduction, if the singularity is mild enough in some sense, then a multiplicity phenomenon appears for $\lambda>0$. Thus, for the model case, the comparison result is sharp in terms of the sign of $\lambda$.
\end{remark}

\begin{proof}[Proof of Theorem \ref{comprinc1}]
Let us denote $w=u-v$. For $k>0$, we consider the function $\phi=(w-k)^+$, and we also denote
\[A_k=\{x\in\Omega:w(x)\geq k\}.\] 
Notice that $\supp(\phi)\subset A_k$. Moreover, condition \eqref{boundarycond} implies that $A_k\subset\subset\Omega$, so $\phi$ has compact support. In particular, $\phi\in H_0^1(\Omega)\cap L^\infty(\Omega)$, so it can be taken as test function in \eqref{compsub} and \eqref{compsuper}, obtaining that

\begin{equation}\label{ineqsub1}
\int_\Omega \nabla u \nabla (w-k)^+ \leq \lambda\int_\Omega u(w-k)^+ +\int_\Omega g(x,u)|\nabla u|^q(w-k)^+ +\int_\Omega h(x)(w-k)^+
\end{equation}
and
\begin{equation}\label{ineqsuper1}
\int_\Omega \nabla v \nabla (w-k)^+ \geq \lambda \int_\Omega v(w-k)^+ +\int_\Omega g(x,v)|\nabla v|^q(w-k)^+ +\int_\Omega h(x)(w-k)^+.
\end{equation}
Subtracting \eqref{ineqsuper1} from \eqref{ineqsub1} we get
\begin{equation*}
\int_\Omega|\nabla(w-k)^+|^2\leq \lambda\int_\Omega((w-k)^+)^2+\lambda k\int_\Omega(w-k)^+ +\int_\Omega(g(x,u)|\nabla u|^q-g(x,v)|\nabla v|^q)(w-k)^+.
\end{equation*}
Since $\lambda\leq 0$, we deduce that
\begin{equation}\label{ineqsubtract}
\int_\Omega|\nabla(w-k)^+|^2\leq\int_\Omega(g(x,u)|\nabla u|^q-g(x,v)|\nabla v|^q)(w-k)^+.
\end{equation}

Assume in order to achieve a contradiction that $w^+\not\equiv 0$, and let $k_0\in (0,\|w^+\|_{L^\infty(\Omega)})$. Let also $\omega\subset\subset\Omega$ be an open set such that $A_{k_0}\subset\omega$. Observe that $A_k\subset A_{k_0}$ for all $k\geq k_0$. Then, using the properties of $g$, it is clear that 
\begin{equation*}
g(x,u)\leq g(x,v)\leq g(x,\inf_\omega(v))\leq\left\|g(\cdot,\inf_\omega(v))\right\|_{L^\infty(\omega)}
\end{equation*}
in $A_k$ for every $k\in [k_0,\|w^+\|_{L^\infty(\Omega)}]$. Therefore, from \eqref{ineqsubtract} we deduce that
\begin{align}\label{ineqgbounded}
\int_\Omega|\nabla(w-k)^+|^2 &\leq\int_\Omega g(x,v)||\nabla u|^q-|\nabla v|^q|(w-k)^+ 
\\
\nonumber &\leq \left\|g(\cdot,\inf_\omega(v))\right\|_{L^\infty(\omega)}\int_{A_k} ||\nabla u|^q-|\nabla v|^q|(w-k)^+
\end{align}
for every $k\in [k_0,\|w^+\|_{L^\infty(\Omega)}]$.

\smallskip

For every $j\in\mathbb{R}$, let us denote $\Omega_j=\{x\in\Omega: |w(x)|=j\}$, and consider also the set $J=\{j\in\mathbb{R}:|\Omega_j|\not =0\}$. Since $|\Omega|<\infty$, then $J$ is at most countable, which implies that the set $\bigcup_{j\in J}\Omega_j$ is measurable, and we also have that

\[
\nabla w=0 \quad\mbox{ in }\bigcup_{j\in J}\Omega_j\implies |\nabla u_1|=|\nabla v_1|\quad\mbox{ in }\bigcup_{j\in J}\Omega_j.\]
Hence, if we define the set $Z=\Omega\setminus \bigcup_{j\in J}\Omega_j$, we deduce from \eqref{ineqgbounded} that 
\begin{equation}\label{pasointermedio}
\int_\Omega |\nabla (w-k)^+|^2 \leq \left\|g(\cdot,\inf_\omega(v))\right\|_{L^\infty(\omega)} \int_{A_k\cap Z} \left(\int_0^1 \frac{d}{dt}(|t\nabla u +(1-t)\nabla v|^q) dt\right) (w-k)^+.
\end{equation}
Taking into account that $u, v\in W^{1,N}_{\mbox{\tiny loc}}(\Omega)$ and $A_k\subset\subset\Omega$, we have that \[|t\nabla u +(1-t)\nabla v|\leq |\nabla u|+|\nabla v|+1\equiv \eta \in  L^N(A_k\cap Z).\] Hence, from \eqref{pasointermedio} we derive that

\begin{align}\label{ineqcontr}
\nonumber\|(w-k)^+\|_{H_0^1(\Omega)}^2 
&\leq    C\int_{A_k\cap Z} \left(\int_0^1 |t\nabla u+(1-t)\nabla v|^{q-2}(t\nabla u+(1-t)\nabla v) \nabla w dt\right) (w-k)^+ 
\\
&\leq   C\int_{A_k\cap Z} \eta^{q-1}  |\nabla w| (w-k)^+\leq   C \int_{A_k\cap Z} \eta |\nabla (w-k)^+| (w-k)^+
\\
\nonumber &\leq    C\|\eta\|_{L^N(A_k\cap Z)} \|(w-k)^+\|_{H_0^1(\Omega)}\|(w-k)^+\|_{L^{2^*}(\Omega)}
\\
\nonumber &\leq   C\|\eta\|_{L^N(A_k\cap Z)} \|(w-k)^+\|_{H_0^1(\Omega)}^2.
\end{align}

Let us now define the function $F:[k_0,\|w^+\|_{L^\infty(\Omega)}]\to \mathbb{R}$ by 
\[
F(k)=\|\eta\|_{L^N(A_k\cap Z)}=\||\nabla u|+|\nabla v|+1\|_{L^N(A_k\cap Z)} \quad \forall k\in [k_0,\|w^+\|_{L^\infty(\Omega)}),\]
and $F(\|w^+\|_{L^\infty(\Omega)})=0$. It is clear that $F$ is nonincreasing and continuous. Thus, choosing $k$ close enough to   $\|w^+\|_{L^\infty(\Omega)}$, we deduce from \eqref{ineqcontr} that $(w-k)^+\equiv 0$. That is to say, $w\leq k$ in $\Omega$. But this is not possible since $k<\|w^+\|_{L^\infty(\Omega)}=\sup_{\Omega} (w)$.

\smallskip

In conclusion, we have proved that $w^+\equiv 0$, i.e., $w\leq 0$ in $\Omega$.
\end{proof}

Next theorem is another comparison principle which works for $\lambda>0$. In turn, one has to impose stronger hypotheses on $g$ and $h$. The proof is similar to the one above combined with some ideas in \cite{CLLM}.

\begin{theorem}\label{comprinc2}

Let $1<q\leq 2$, $\lambda\in\mathbb{R}$, $0\leq h\in L^1_{\mbox{\tiny loc}}(\Omega)$ and $g:\Omega\times (0,+\infty)\to [0,+\infty)$ satisfying
\begin{align*}
s&\mapsto s^{q-1}g(x,s)\,\,\text{ is nonincreasing for a.e. }x\in\Omega,
\\
x&\mapsto g(x,s)\,\,\,\,\quad\quad\text{is locally essentially bounded for all }s>0.
\end{align*}
If $\lambda>0$, assume also that
\begin{equation}\label{hcond}
\forall\omega\subset\subset\Omega\quad\exists c_\omega>0:\quad h\geq c_\omega\text{ in }\omega.
\end{equation}
Let $u, v\in C(\Omega)\cap W^{1,N}_{\mbox{\tiny loc}}(\Omega)$, with $u,v>0$ in $\Omega$, satisfying respectively \eqref{compsub} and \eqref{compsuper} for every $0\leq \phi\in H_0^1(\Omega)\cap L^\infty(\Omega)$ with compact support. 
Suppose also that, for every $\varepsilon>0$, the following boundary condition holds:
\begin{equation}\label{wellor:eps}
\limsup_{x\to x_0}\left(\frac{u(x)}{v(x)+\varepsilon}\right)\leq 1 \quad \forall x_0\in\partial\Omega.
\end{equation}
Then, $u\leq v$ in $\Omega$. 
\end{theorem}

\begin{proof}
For every $\varepsilon>0$, let us consider the function 
\[w_\varepsilon=\log\left(\frac{u}{v+\varepsilon}\right).\]
We claim that $w_\varepsilon^+\equiv 0$ for any $\varepsilon>0$. Suppose by contradiction that there exists  $ \varepsilon_0>0$ such that $w^+_{\varepsilon_0} \not\equiv 0 $. Let us fix $k_0\in \left(0,\|w_{\varepsilon_0}^+\|_{L^\infty(\Omega)}\right)$ and $\varepsilon \in (0,\varepsilon_0)$, the latter to be chosen small enough later. It is clear that $w_{\varepsilon_0}\leq w_\varepsilon$ in $\Omega$, so $w_\varepsilon^+\not\equiv 0$.

\smallskip

For $k\in [k_0,\|w_\varepsilon^+\|_{L^\infty(\Omega)}]$, let us denote \[A_k=\{x\in\Omega: w_\varepsilon(x)\geq k\}=\{x\in\Omega:u(x)\geq e^k (v(x)+\varepsilon)\}.\] Notice that $\supp(w_\varepsilon-k)^+\subset A_k$. By \eqref{wellor:eps}, we also have that $\displaystyle \limsup_{x\to x_0}w_\varepsilon(x)\leq 0$ for all $x_0\in \partial\Omega$, which implies that $A_k\subset\subset\Omega$. Then, the function $(w_\varepsilon-k)^+$ has compact support, and in particular, $(w_\varepsilon-k)^+\in H_0^1(\Omega)\cap L^\infty(\Omega)$. Therefore, we may take $\frac{(w_\varepsilon-k)^+}{u}$ as test function in \eqref{compsub}, and $\frac{(w_\varepsilon-k)^+}{v+\varepsilon}$ in \eqref{compsuper}, obtaining
\begin{align}\label{ineqsub2}
\nonumber\int_\Omega \frac{\nabla u}{u} \nabla (w_\varepsilon-k)^+ &\leq \int_\Omega \frac{|\nabla u|^2}{u^2}(w_\varepsilon-k)^+ +\lambda \int_\Omega (w_\varepsilon-k)^+
\\
&+\int_\Omega u^{q-1}g(x,u)\frac{|\nabla u|^q}{u^q}(w_\varepsilon-k)^+ +\int_\Omega \frac{h(x)}{u}(w_\varepsilon-k)^+
\end{align}
and, using that $g\geq 0$,
\begin{align}
\label{ineqsuper2}
\nonumber\int_\Omega \frac{\nabla v}{v+\varepsilon} \nabla (w_\varepsilon -k)^+ &\geq \int_\Omega \frac{|\nabla v|^2}{(v+\varepsilon)^2}(w_\varepsilon -k)^+ +\lambda \int_\Omega \frac{v}{v+\varepsilon}(w_\varepsilon-k)^+
\\
\nonumber&+\int_\Omega v^{q-1}g(x,v)\frac{|\nabla v|^q}{v^{q-1}(v+\varepsilon)}(w_\varepsilon-k)^+ +\int_\Omega \frac{h(x)}{v+\varepsilon}(w_\varepsilon-k)^+
\\
&\geq \int_\Omega \frac{|\nabla v|^2}{(v+\varepsilon)^2}(w_\varepsilon -k)^+ +\lambda \int_\Omega (w_\varepsilon-k)^+ - \int_\Omega \frac{\lambda\varepsilon}{v+\varepsilon}(w_\varepsilon-k)^+
\\
\nonumber&+\int_\Omega v^{q-1}g(x,v)\frac{|\nabla v|^q}{(v+\varepsilon)^q}(w_\varepsilon-k)^+ +\int_\Omega \frac{h(x)}{v+\varepsilon}(w_\varepsilon-k)^+.
\end{align}

Let $\omega\subset\subset\Omega$ be an open set such that $A_{k_0}\subset\omega$. Observe that $A_k\subset A_{k_0}$ for all $k\geq k_0$. Then, it is clear that 
\begin{equation*}
u^{q-1}g(x,u)\leq v^{q-1}g(x,v)\leq \sup_\omega(v)^{q-1}g(x,\inf_\omega(v))\leq\sup_\omega(v)^{q-1}\left\|g(\cdot,\inf_\omega(v))\right\|_{L^\infty(\omega)}
\end{equation*}
in $A_k$ for every $k\in [k_0,\|w_\varepsilon^+\|_{L^\infty(\Omega)}]$. Therefore,
\begin{align*}
&\int_\Omega \left(u^{q-1}g(x,u)\frac{|\nabla u|^q}{u^q} -v^{q-1}g(x,v)\frac{|\nabla v|^q}{(v+\varepsilon)^q}\right)(w_\varepsilon-k)^+ 
\\
&\leq\sup_\omega(v)^{q-1}\left\|g(\cdot,\inf_\omega(v))\right\|_{L^\infty(\omega)}\int_\Omega\left|\frac{|\nabla u|^q}{u^q} -\frac{|\nabla v|^q}{(v+\varepsilon)^q}\right|(w_\varepsilon-k)^+.
\end{align*}

Moreover, we have that
\begin{equation}\label{hineq}
h\left(\frac{1}{u}-\frac{1}{v+\varepsilon}\right)+\frac{\lambda\varepsilon}{v+\varepsilon}\leq 0 \quad\mbox{ in } A_k \mbox{ for every }k\in [k_0,\|w_\varepsilon^+\|_{L^\infty(\Omega)}]
\end{equation}
whenever $\lambda\leq 0$. On the other hand, if $\lambda>0$, let us take 
\[\varepsilon<\min\left\{\varepsilon_0,\frac{1-e^{-k_0}}{\lambda}\ c_\omega\right\},\]
where $c_\omega$ is the constant given by \eqref{hcond}. With this choice, it is straightforward to deduce that \eqref{hineq} holds again.

\smallskip

Therefore, subtracting \eqref{ineqsub2} and \eqref{ineqsuper2}, and taking into account that $u,v\in W_{\mbox{\tiny loc}}^{1,N}(\Omega)$ and also \eqref{hineq}, we may argue as in the proof of \cite[Theorem 3.2]{CLLM} and achieve a contradiction taking $k$ close enough to $\|w_\varepsilon^+\|_{L^\infty(\Omega)}$.

\smallskip

In conclusion, necessarily $w_\varepsilon^+\equiv 0$ for any $\varepsilon>0$, i.e., $u\leq v+\varepsilon$ in $\Omega$ for any $\varepsilon>0$. Letting $\varepsilon\to 0$  it follows that $u\leq v$ in $\Omega$.
\end{proof}

\section{Multiplicity for $0\leq\alpha<q-1$}\label{sec:superlinear}

In this section we will study problem \eqref{problem} under condition \eqref{H1}. In this case observe that, if $0<u\in W_{\mbox{\tiny loc}}^{1,1}(\Omega)$ and $t>0$, then
\[\frac{|\nabla tu|^q}{(tu)^\alpha}=t^{q-\alpha}\frac{|\nabla u|^q}{u^\alpha}.\]
Since $\alpha<q-1$, then $q-\alpha>1$. That is to say, the lower order term has \emph{superlinear homogeneity}. 

\smallskip

The concept of solution we will adopt is gathered in the following definition.

\begin{definition}\label{solution}
Given $\lambda\in\re$, a subsolution to \eqref{problem} is a function $u\in H^1_0(\Omega)\cap L^\infty(\Omega)$ such that $u>0$ a.e. in $\Omega$, $\mu\frac{|\nabla u|^q}{u^\alpha}\in L^1_{\mbox{\tiny loc}}(\Omega)$ and 
\[\int_\Omega\nabla u\nabla \phi\leq \lambda\int_\Omega u\phi+\int_\Omega\mu(x)\frac{|\nabla u|^q}{u^\alpha}\phi+\int_\Omega f(x)\phi\quad\forall 0\leq\phi\in C_c^1(\Omega).\]
Reciprocally, a supersolution to \eqref{problem} is  a function $u\in H^1(\Omega)\cap L^\infty(\Omega)$ such that $u>0$ a.e. in $\Omega$, $\mu\frac{|\nabla u|^q}{u^\alpha}\in L^1_{\mbox{\tiny loc}}(\Omega)$ and satisfies the reverse inequality.
Finally, a solution to \eqref{problem} is a function $u\in H_0^1(\Omega)\cap L^\infty(\Omega)$ which is both a subsolution and a supersolution to \eqref{problem}.
\end{definition}

\begin{remark}\label{equivalentdef}
Arguing as in \cite[Appendix]{CLLM}, it can be proved that a definition of subsolution, supersolution and solution to \eqref{problem} using test functions in $H_0^1(\Omega)\cap L^\infty(\Omega)$ is equivalent to Definition \ref{solution}. Moreover, even the concepts of supersolution and solution to problem \eqref{problem} with test functions only in $H_0^1(\Omega)\cap L^\infty(\Omega)$ are equivalent to the corresponding concepts in Definition \ref{solution}.
\end{remark}

\begin{remark}\label{lambda1remark1}
Assume that \eqref{H1} holds. By taking $\varphi_1$ as test function in the weak formulation of \eqref{problem} one easily deduces that, if $u$ is a solution to \eqref{problem}, then $\lambda<\lambda_1$. Furthermore, since $\alpha\in [0,1]$, it can be proved as in \cite[Appendix]{CLLM}, which follows the ideas in \cite{LU}, that every solution $u$ to \eqref{problem}, for any $\lambda<\lambda_1$, satisfies that $u\in C^{0,\eta}(\overline{\Omega})$ for some $\eta\in (0,1)$. Finally, since the solutions to \eqref{problem} are positive in compact subsets of $\Omega$, then it can be seen again as in the mentioned appendix  that $u\in W^{1,N}_{\mbox{\tiny loc}}(\Omega)$ for every solution to \eqref{problem} for any $\lambda<\lambda_1$. 
\end{remark}

Our first result is concerned with the existence and uniqueness of solution to \eqref{problem} for $\lambda\leq 0$. The existence is well-known from the works that are quoted in the proof below. However, a precise statement for unbounded datum $f$ is required for our purposes. In any case, the uniqueness is new up to our knowledge.

\begin{proposition}\label{exislambdaneg}
Assume that \eqref{H1} holds. Then, problem \eqref{problem} has a unique  solution for all $\lambda<0$. Moreover, assume additionally that either $\alpha>0$ or the following smallness condition holds:
\[a\left(b+\|f\|_{L^{p_0}(\Omega)}\right)<\left(\frac{2}{N}-\frac{1}{p_0}\right)\frac{N^2 |B_1(0)|^\frac{2}{N}}{|\Omega|^{\frac{2}{N}-\frac{1}{p_0}}},\]
where $B_1(0)$ denotes the unit ball in $\mathbb{R}^N$, and $a,b>0$ are such that
\[\|\mu\|_{L^\infty(\Omega)}|s|^q\leq a|s|^2+b\quad\forall s\in\mathbb{R}.\] Then $(P_0)$ has a unique solution. 
\end{proposition}

\begin{proof}

The result for $\alpha=0$ and $\lambda\leq 0$ is well-known. Indeed, the existence of solution for $\lambda<0$ is proved in \cite{BMP1, BMP2}, the existence for $\lambda=0$ under the smallness condition is proved in \cite{FPR}, and the uniqueness for $\lambda\leq 0$, in \cite{ACJT2}. Thus, we assume that $\alpha\in(0,q-1)$.

\smallskip

Observe now that, by Young's inequality, there exist $C_1,C_2>0$ such that
\begin{equation}\label{GPSineq}
0\leq\mu(x)\frac{|\xi|^q}{|s|^\alpha}\leq C_1\frac{|\xi|^2}{|s|^{\frac{2\alpha}{q}}}+C_2
\end{equation}
for all $\xi\in\re^N$, for all $s\in\re\setminus\{0\}$ and for a.e. $x\in\Omega$, where 
\begin{equation}\label{GPScond}
\frac{2\alpha}{q}<\frac{2(q-1)}{q}=2-\frac{2}{q}< 1.
\end{equation}
Then, the hypotheses of \cite[Proposition 4.1]{GPS} are fulfilled, so there exists a solution $u_0\in H_0^1(\Omega)\cap L^\infty(\Omega)$ to $(P_0)$ in some weaker sense than Definition \ref{solution}. Nonetheless, since $f\gneqq 0$ in $\Omega$, then the strong maximum principle implies that $u_0>0$ in $\Omega$, so $u_0$ is in fact a solution to \eqref{problem} in the sense of Definition \ref{solution}.

\smallskip

Concerning the existence for $\lambda<0$, we argue by approximation as follows. For all $n\in\na$, let us consider the problem
\begin{equation}\label{approxprob1}
\begin{cases}
\dys-\Delta u_n=\lambda u_n+\mu(x)\frac{|\nabla u_n|^q}{u_n^\alpha}+T_n(f(x))\quad &\text{ in }\Omega,
\\
u_n>0 \quad &\text{ in }\Omega,
\\
u_n=0\quad &\text{ on }\partial\Omega.
\end{cases}
\end{equation}
Since \eqref{GPSineq} and \eqref{GPScond} hold, we know from \cite{GM} that there exists a solution $u_n\in H_0^1(\Omega)\cap L^\infty(\Omega)$ to \eqref{approxprob1} for all $n$. Notice now that 
\[-\Delta u_n\leq\mu(x)\frac{|\nabla u_n|^q}{u_n^\alpha}+f(x)\text{ in }\Omega.\]
Hence, Theorem \ref{comprinc1} applies (see Remark \ref{lambda1remark1}) and yields
\[u_n\leq u_0\leq \|u_0\|_{L^\infty(\Omega)}\text{ in }\Omega.\]
In other words, $\{u_n\}$ is bounded in $L^\infty(\Omega)$. By taking $u_n$ as test function in the weak formulation of \eqref{approxprob1}, we immediately deduce that $\{u_n\}$ is also bounded in $H_0^1(\Omega)$. Hence, there exists $u\in H_0^1(\Omega)\cap L^\infty(\Omega)$ such that, passing to a subseqence, $u_n\rightharpoonup u$ weakly in $H_0^1(\Omega)$ and $u_n\to u$ strongly in $L^p(\Omega)$ for any $p\in [1,\infty)$.

\smallskip

Observe also that, again by comparison, $u_n\geq z$ for all $n$, where $z\in H_0^1(\Omega)\cap L^\infty(\Omega)$ is the unique solution to
\begin{equation*}
\begin{cases}
-\Delta z=\lambda z+T_1(f(x))\quad&\text{in }\Omega,
\\
z=0\quad&\text{on }\partial\Omega.
\end{cases}
\end{equation*}
Now, the strong maximum principle applied on $z$ implies that
\[\forall\omega\subset\subset\Omega\quad\exists c_\omega>0:\quad u_n\geq c_\omega\quad\forall n.\]
Therefore, $\{-\Delta u_n\}$ is bounded in $L^1_{\mbox{\tiny loc}}(\Omega)$. Thus, by virtue of \cite[Theorem 2.1]{BM}, $\nabla u_n\to\nabla u$ strongly in $L^q(\Omega)^N$, up to a subsequence. The convergences we have proved about $\{u_n\}$ and $\{\nabla u_n\}$ are enough to pass to the limit in \eqref{approxprob1}. The proof is standard, we refer to the proof of \cite[Proposition 5.2]{CLLM} for further details. In sum, $u$ is a solution to \eqref{problem}.

\smallskip 

The uniqueness of $u$ is a direct consequence of Theorem \ref{comprinc1} and Remark \ref{lambda1remark1}.
\end{proof}

Next result shows that, if $\alpha=0$, then the existence of solution to $(P_0)$ may fail if $f$ or $\mu$ are too large in some sense, in contrast to the case $\alpha>0$. Thus, the smallness assumption in Proposition \ref{exislambdaneg} is justified. This result is basically contained in \cite[Theorem 2.1]{AP}. We include the statement and proof in our context for completeness.

\begin{proposition}\label{nonexist}
Assume that \eqref{H1} holds with $\alpha=0$, and suppose that \eqref{problem} admits a solution for some $\lambda\geq 0$. Then,
\[\int_\Omega f(x)\phi^{q'}\leq\int_\Omega\frac{|\nabla\phi|^{q'}}{((q-1)\mu(x))^\frac{1}{q-1}}\quad\forall 0\leq\phi\in W_0^{1,q'}(\Omega)\cap L^\infty(\Omega).\]
\end{proposition}

\begin{proof}
Let $u$ be a solution to \eqref{problem}, and let $0\leq\phi\in W_0^{1,q'}(\Omega)\cap L^\infty(\Omega)$. Since $q'>2$, then $\phi^{q'}\in H_0^1(\Omega)\cap L^\infty(\Omega)$, so it can be taken as test function in the weak formulation of \eqref{problem} to obtain, after using Young's inequality, that
\begin{align*}
\int_\Omega\left(\lambda u+\mu(x)|\nabla u|^q+f(x)\right)\phi^{q'}&=\int_\Omega\nabla u\nabla(\phi^{q'})=q'\int_\Omega\phi^{q'-1}\nabla u\nabla\phi
\\
&\leq\int_\Omega\mu(x) |\nabla u|^q \phi^{q'}+\int_\Omega\frac{|\nabla \phi|^{q'}}{((q-1)\mu(x))^\frac{1}{q-1}}.
\end{align*}
Hence, it is now clear that the result follows.
\end{proof}

Our aim in the next two subsections is to prove, for a fixed $\lambda_0>0$, an $L^\infty$ estimate for the solutions to \eqref{problem} for all $\lambda>\lambda_0$. Such an estimate implies that zero is the only possible bifurcation point from infinity to problem \eqref{problem}. This fact will be the key to prove multiplicity of solutions to \eqref{problem} for $\lambda>0$ small enough.

\subsection{A priori $L^p$ estimates}\text{}

\smallskip

This subsection is devoted to proving an $L^p$ estimate on the supersolutions to \eqref{problem} for $\lambda>0$. The techniques employed here have been taken from \cite{Souplet}.

\smallskip

The first result of the subsection provides an apparently weak local estimate on the solutions to \eqref{problem}. Notwithstanding, this is the starting point for proving the $L^\infty$ estimate we are aiming at. Concerning the proof, we will argue similarly as in Proposition \ref{nonexist}.

\begin{lemma}\label{locallemma}
Assume that \eqref{H1} holds. Then, for every $\lambda_0>0$ and $\omega\subset\subset\Omega$ there exists $C>0$ such that 
\begin{equation}\label{localbound}
\int_\omega u \leq C.
\end{equation}
for every supersolution $u$ to \eqref{problem} with $\lambda>\lambda_0$. 
\end{lemma}

\begin{proof}
Let $\phi\in C_c^1(\Omega)$ be such that $\omega\subset\subset\supp{(\phi)}$, $0\leq\phi\leq 1$ in $\Omega$ and $\phi=1$ in $\omega$. Taking $\phi^{\beta}\in C_c^1(\Omega)$ for some $\beta>1$ as test function in \eqref{problem} and using Young's inequality twice we obtain that
\begin{align*}
\int_\Omega\left(\lambda u+\mu(x)\frac{|\nabla u|^q}{u^\alpha}+f(x)\right)\phi^{\beta}&\leq\int_\Omega\nabla u\nabla(\phi^{\beta})=\beta\int_\Omega\phi^{\beta-1}\nabla u\nabla\phi
\\
&\leq\frac{\mu_0}{2}\int_\Omega\frac{|\nabla u|^q}{u^\alpha}\phi^{\beta}+C\int_\Omega\frac{|\nabla(\phi^\beta)|^{q'}}{\phi^{\beta(q'-1)}}u^{\frac{\alpha}{q-1}}
\\
&\leq \frac{\mu_0}{2}\int_\Omega\frac{|\nabla u|^q}{u^\alpha}\phi^{\beta}+\frac{\lambda_0}{2}\int_\Omega u\phi^\beta + C\int_\Omega\left(\frac{|\nabla\phi|}{\phi}\right)^{\frac{q}{q-1-\alpha}}\phi^\beta.
\end{align*}
Taking $\beta=\frac{q}{q-1-\alpha}$, the last term in the previous inequality is bounded. Therefore, 
\begin{equation*}
\int_\Omega\left(\lambda_0 u+\mu_0\frac{|\nabla u|^q}{u^\alpha}+f(x)\right)\phi^\beta\leq \frac{\mu_0}{2}\int_\Omega\frac{|\nabla u|^q}{u^\alpha}\phi^{\beta}+\frac{\lambda_0}{2}\int_\Omega u\phi^\beta + C,
\end{equation*}
so \eqref{localbound} follows by taking into account that $\phi=1$ in $\omega$.
\end{proof}

The following is a slightly more general version of \cite[Lemma 3.2]{BC}.

\begin{lemma}\label{BClemma}
Let $\Omega\subset\mathbb{R}^N$ be a bounded domain with boundary of class $\mathcal{C}^2$, and let $0\leq h\in L^1(\Omega)$ and $v\in H^1(\Omega)$ be such that $v^-\in H_0^1(\Omega)$ and $-\Delta v\geq h$ in $\Omega$. Then, there exists a constant $C>0$ depending only on $\Omega$ such that
\[\frac{v}{\delta}\geq C \int_\Omega \delta h\quad \text{ a.e. in }\Omega.\]
\end{lemma}

\begin{proof}
Let us consider the following problem for all $n\in\mathbb{N}$:

\begin{equation*}
\begin{cases}
-\Delta v_n=T_n(h(x)),&x\in\Omega,
\\
v_n=0, &x\in\partial\Omega,
\end{cases}
\end{equation*}
It is well-known that it has a unique solution $v_n\in C_0^{1,\nu}(\overline{\Omega})$ for all $\nu\in (0,1)$. Moreover, \cite[Lemma 3.2]{BC} implies that
\[v_n(x)\geq C\delta(x)\int_\Omega\delta T_n(h)\quad \forall x\in\Omega,\]
for some $C>0$ depending only on $\Omega$. In particular, it does not depend on $n$.

On the other hand, by comparison, it is clear that $v_n\leq v$ a.e. in $\Omega$, so
\[v\geq C\delta\int_\Omega\delta T_n(h)\quad\text{a.e. in }\Omega.\]
We conclude the proof by letting $n$ tend to infinity.
\end{proof}

Next lemma is an immediate consequence of Lemma \ref{BClemma}.

\begin{lemma}\label{BCconsequence}
Assume that \eqref{H1} holds. Then, there exists $C>0$ such that
\begin{equation}\label{breziscabrebound}
u(x)\geq C\delta(x)\int_\Omega\left(\lambda u+\mu(y)\frac{|\nabla u|^q}{u^\alpha}+f(y)\right)\delta(y)dy\quad\text{a.e. } x\in\Omega,
\end{equation}
for every supersolution $u$ to \eqref{problem} with $\lambda> 0$.
\end{lemma}

Combining Lemmas \ref{locallemma} and \ref{BCconsequence} we obtain in the following result some estimates in weighted Lebesgue spaces. 

\begin{lemma}\label{lpdeltaestimates}
Assume that \eqref{H1} holds. Then, for every $\lambda_0>0$ there exists $C>0$ such that
\begin{enumerate}
\item\label{ubound}$\|u\|_{L^{p,\delta}(\Omega)}\leq C\quad\forall p\in\left[1,\frac{N+1}{N-1}\right)$,
\\
\item\label{gradbound} $\left\|\frac{|\nabla  u|^q}{u^{\alpha}}\right\|_{L^{1,\delta}(\Omega)}=C\||\nabla u^{1-\frac{\alpha}{q}}|\|_{L^{q,\delta}(\Omega)}\leq C$,
\end{enumerate}
for every supersolution $u$ to \eqref{problem} with $\lambda>\lambda_0$.
\end{lemma}

\begin{proof}
Integrating both sides of inequality \eqref{breziscabrebound} over any open set $\omega\subset\subset\Omega$ and using the estimate \eqref{localbound} we deduce that
\[\int_\Omega(-\Delta u)\delta=\int_\Omega\left(\lambda u+\mu(x)\frac{|\nabla u|^q}{u^\alpha}+f(x)\right)\delta\leq C\left(\int_\omega u\right)\left(\int_\omega\delta\right)^{-1}\leq C.\]
In particular,
\[\int_\Omega\frac{|\nabla u|^q}{u^\alpha}\delta\leq C,\]
and this is equivalent to item \eqref{gradbound}. Regarding item \eqref{ubound}, observe that 
\[\|\Delta u\|_{L^{1,\delta}(\Omega)}\leq C.\]
Hence, by \cite[Proposition 2.2]{FSW} we obtain directly item \eqref{ubound}.
\end{proof}

We finish the subsection with the best $L^p$ estimate for supersolutions that we obtain with these techniques.

\begin{lemma}\label{lebesguelemma}
Assume that \eqref{H1} holds. Then, for every $\lambda_0>0$ there exists $C>0$ such that
\begin{equation}\label{lebesgueestimate}
\|u\|_{L^m(\Omega)}\leq C
\end{equation}
for every supersolution $u$ to \eqref{problem} with $\lambda>\lambda_0$, where $m=\frac{(q-\alpha)N}{N-q+1}\in (q-\alpha,(q-\alpha)^*)$.
\end{lemma}

\begin{proof}

Let us denote $v=u^{1-\frac{\alpha}{q}}$. Since $1-\frac{\alpha}{q}>\frac{1}{2}$, we can argue as in \cite[Lemma 2.6]{CLLM} to prove that $v\in H_0^1(\Omega)$. Then, \cite[Proposition 2]{Souplet} implies that
\begin{equation*}
\int_\Omega v^q \delta^{-(q-1)}\leq C\left(\int_\Omega v\delta\right)^q+C\left(\int_\Omega|\nabla v|^q\delta\right),
\end{equation*}
and 
\begin{equation*}
\left(\int_\Omega v^{q^*}\delta^\frac{N}{N-q}\right)^{q/q^*}\leq C\left(\int_\Omega  v\delta\right)^q+C\left(\int_\Omega|\nabla v|^q\delta\right).
\end{equation*}
Hence, by Lemma \ref{lpdeltaestimates} we derive that
\begin{equation}\label{deltaqbounds}
\int_\Omega v^q \delta^{-(q-1)}\leq C\quad\text{ and }\quad\int_\Omega v^{q^*}\delta^\frac{N}{N-q}\leq C.
\end{equation}
Now, \cite[Lemma 3]{Souplet} implies that

\begin{equation}\label{gammabound}
\int_\Omega v^b\delta^\gamma\leq C\left(\int_\Omega v^q\delta^{-(q-1)}\right)^\theta\left(\int_\Omega v^{q^*}\delta^\frac{N}{N-q}\right)^{1-\theta},
\end{equation}
where
\[b=\frac{qN}{N-q+1},\quad \theta=\frac{q^*-b}{q^*-q}\in (0,1)\quad\text{ and }\quad\gamma=\frac{N}{N-q}-\frac{(q^*-b)(q-1+\frac{N}{N-q})}{q^*-q}.\]
It is easy to check that, in fact, $\gamma=0$. Therefore, recalling that $m=b\left(1-\frac{\alpha}{q}\right)$, by \eqref{gammabound} and \eqref{deltaqbounds} we conclude that
\[\int_\Omega v^b=\int_\Omega u^m\leq C,\]
and the result holds true.
\end{proof}

\subsection{A priori $L^\infty$ estimates}\text{}

\smallskip

In this subsection we will show how to obtain $L^\infty$ estimates on the solutions to \eqref{problem} for $\lambda>0$ by combining the $L^p$ estimate given by Lemma \ref{lebesguelemma} and a bootstrapp argument. We will make use of several results in \cite{GMP}. In fact, the ideas in such a paper will be used also to derive some new results which provide analogous estimates in our singular framework.

\smallskip

We start the subsection with the easier case $\alpha=0$, which is interesting itself; we will deal with the singular case $\alpha\in (0,q-1)$ later. Thus we state and prove the following

\begin{proposition}\label{gmpbootstrap1}
Assume that \eqref{H1} holds with $\alpha=0$, and consider the sequence $\{Q_n\}$ defined by \eqref{Qn}, i.e.,
\begin{equation*}
Q_n=\left\{
\begin{array}{ll}
\displaystyle 2&\forall n\leq 4,
\\
\displaystyle\frac{n+2-\sqrt{n^2-4n-4}}{4}&\forall n\geq 5.
\end{array}
\right.
\end{equation*}
Then, for every $q\in (1,Q_N]\setminus\{2\}$ and every $\lambda_0>0$, there exists $C>0$ such that
\begin{equation}\label{maximumestimate2}
\|u\|_{L^\infty(\Omega)}\leq C
\end{equation}
for every solution $u$ to \eqref{problem} with $\lambda>\lambda_0$.
\end{proposition}

\begin{proof}
In this proof, $C$ denotes a positive constant independent of $u$ and $\lambda$ whose value may vary from line to line.

\smallskip
We start by assuming that $1<q<\frac{N}{N-1}$. Observe that $\frac{N}{N-1}<Q_N$, so $q\leq Q_N$ is not a restriction in this case.

\smallskip

Let us denote $h(x)=(\lambda +1)u+f(x)$. Then, $u$ satisfies 
\begin{equation*}
\begin{cases}
u-\Delta u=\mu(x)|\nabla u|^q+h(x),\quad&x\in\Omega,
\\
u=0,\quad&x\in\partial\Omega.
\end{cases}
\end{equation*} 

We know from Lemma \ref{lebesguelemma} that $\|u\|_{L^m(\Omega)}\leq C$, where $m=\frac{(q-\alpha)N}{N-q+1}$, so $\|h\|_{L^p(\Omega)}\leq C$, where $p=\min\{m,p_0\}$. If $m>\frac{N}{2}$, and taking into account that $p_0>\frac{N}{2}$, then \cite[Theorem 5.8, item (i)]{GMP} implies that $\|u\|_{L^\infty(\Omega)}\leq C$. 

\smallskip

Let us assume now that $m=\frac{N}{2}$. Then, \cite[Theorem 5.8, item (ii)]{GMP} implies that $\|u\|_{L^p(\Omega)}\leq C$ for all $p<\infty$. In particular, $\|h\|_{L^{p_0}(\Omega)}\leq C$. Since $p_0>\frac{N}{2}$, then again item (i) of the same mentioned theorem yields the $L^\infty$ estimate.

\smallskip

Suppose now that $(2^*)'<m<\frac{N}{2}$. Let us define the sequence $\{m_n\}$ inductively as
\[m_n=m_{n-1}^{**}=\frac{N m_{n-1}}{N-2 m_{n-1}}\quad\forall n\in\mathbb{N},\]
where $m_0=m$. This is clearly an increasing sequence. Moreover, using one more time \cite[Theorem 5.8, item (iii)]{GMP}, it is easy to see that $\|u\|_{L^{m_n}(\Omega)}\leq C$ for $n\in\mathbb{N}$ as long as $m_n<\frac{N}{2}$. In particular, the same holds for $h$. 

\smallskip

Assume by contradiction that $m_n<\frac{N}{2}$ for all $n\in\mathbb{N}$. Since $\{m_n\}$ is increasing and bounded from above, there exists $l\leq\frac{N}{2}$ such that, passing to a not relabeled subsequence, $m_n\to l$. Consequently,
\[l=\frac{Nl}{N-2l}.\]
From this equality we deduce that $l=0$. But this is a contradiction because $m_0>0$ and the sequence is increasing. Therefore, $m_n\geq\frac{N}{2}$ for some $n\in\mathbb{N}$, so the previous cases imply that $\|u\|_{L^\infty(\Omega)}\leq C$.

\smallskip

It only remains to consider the case $1<m\leq (2^*)'$. Now, item (iv) of the same theorem implies that 
\[\|(1+u)^{\tau-1}u\|_{L^{2^*}(\Omega)}\leq C,\quad\text{where}\quad \tau=\frac{m (N-2)}{2(N-2m)}=\frac{m^{**}}{2^*}\leq 1.\]
On the other hand, it is straightforward to prove that, for any $a\in (0,1)$, there exists a constant $b>0$ such that
\[as^{\tau}\leq \frac{s}{(1+s)^{1-\tau}}+b\quad\forall s\geq 0.\]
Then, with $m_n=m_{n-1}^{**}$ and $m_0=m$, as before,
\[\|u\|_{L^{m_1}(\Omega)}=\|u\|_{L^{2^*\tau}(\Omega)}\leq C(\|(1+u)^{\tau-1}u\|_{L^{2^*}(\Omega)}+1)\leq C.\]
In particular, $\|h\|_{L^{m_1}(\Omega)}\leq C$. It can be proved inductively that $\|u\|_{L^{m_n}(\Omega)}\leq C$ as long as $m_n\leq(2^*)'$. Arguing as above, we deduce that $\{m_n\}$ is increasing and divergent. Hence, $m_n>(2^*)'$ for some $n\in\mathbb{N}$, and the proof concludes using the previous cases.

\smallskip

We now turn to the range $\frac{N}{N-1}<q<2$. The procedure is the same as above, but in this case, instead of Theorem 5.8,  one has to apply (a finite number of times) either \cite[Theorem 4.9]{GMP} or \cite[Theorem 3.8]{GMP}, depending on the value of $q$. In both cases, one has to verify in the first step of the bootstrap that $h\in L^{\frac{(q-1)N}{q}}(\Omega)$ so that the hypotheses of both theorems are satisfied. We know by virtue of Lemma \ref{lebesguelemma} that $h\in L^m(\Omega)$, so we have to impose that
\[\frac{N(q-1)}{q}\leq \frac{qN}{N-q+1}.\]
One can easily check that the previous inequality is satisfied if and only if $q\leq Q_N$.

\smallskip

It is left to consider the case $q=\frac{N}{N-1}$. Since $\frac{N}{N-1}<Q_N$, we can take $\varepsilon>0$ small enough so that
\[\frac{N}{N-1}<q+\varepsilon<Q_N.\]
Moreover, we have by Young's inequality that
\[\mu(x)|\xi|^q+h(x)\leq \mu(x)|\xi|^{q+\varepsilon}+h_\varepsilon(x)\quad\forall \xi\in\mathbb{R}^N,\text{ a.e. }x\in\Omega,\]
where $h(x)=(\lambda +1) u+f(x)$ and $h_\varepsilon(x)=h(x)+C_\varepsilon$ for some $C_\varepsilon>0$. Therefore, the previous case can be applied and the proof concludes.
\end{proof}

We deal now with the singular case. For this purpose, it is necessary to derive results similar to the ones from \cite{GMP} mentioned in the previous proof, but valid for singular equations. Even though our results are not proper extensions in the whole generality (as in \cite{GMP} the solutions are weaker than ours and the terms in their equation are not explicit and only satisfy growth restrictions), they are new in considering singular terms. 

\smallskip

The mentioned results will be concerned with the following auxiliary problem:

\begin{equation}\label{auxproblem}
\begin{cases}
\dys \beta u-\Delta u=\mu(x)\frac{|\nabla u|^q}{u^\alpha}+h(x)\quad &\text{ in }\Omega,
\\
u>0 \quad &\text{ in }\Omega,
\\
u=0\quad &\text{ on }\partial\Omega,
\end{cases}
\end{equation}
where the parameters satisfy

\begin{equation}\label{parameters}
1<q<2,\quad \alpha\in [0,q-1),\quad \beta>0, \quad 0\lneqq \mu\in L^\infty(\Omega).
\end{equation}

Let us define the functions $r,\sigma:[0,q-1)\to\re$ as

\begin{equation}\label{keyparameters}
r(\theta)=\frac{N(q-1-\theta)}{q-2\theta},\quad \sigma(\theta)=\frac{(N-2)(q-1-\theta)}{2(2-q)}\quad\forall\theta\in [0,q-1).
\end{equation}
Observe also that both $r,\sigma$ are decreasing functions. We will not write their dependence on $\theta$ when confusion is not arisen. 

\smallskip

The following result provides estimates on solutions to \eqref{auxproblem} when $q$ is large and $h$ has enough summability. 

\begin{proposition}\label{estimatethm1}
Assume that $q,\alpha,\beta,\mu$ satisfy \eqref{parameters}, let $r,\sigma$ be defined as in \eqref{keyparameters} and let $h\in L^p(\Omega)$ with $p>1$. Assume also that 
\begin{equation*}
1+\frac{2}{N}\leq q<2,
\end{equation*}
and let us denote
\[a(q)=\left(q-1-\frac{2}{N}\right)\frac{N}{N-2}\in (0,q-1).\]
Then, for all $C>0$ and $p>1$, there exists $M>0$ such that, for any $h\in L^p(\Omega)$ with $\|h\|_{L^p(\Omega)}\leq~C$ and for any solution $u$ to problem \eqref{auxproblem}, the following holds:
\begin{enumerate}
\item If $p=r(\theta)$ for some $0\leq\theta\leq\min\left\{\alpha,a(q)\right\}$, then $\|u\|_{H_0^1(\Omega)}+\|u^{\sigma(\theta)}\|_{H_0^1(\Omega)}\leq M$;
\\
\item if $r(0)<p<\frac{N}{2}$, then $\|u^\tau\|_{H_0^1(\Omega)}\leq M$, where $\tau=\frac{p(N-2)}{2(N-2p)}$;
\\
\item if $p=\frac{N}{2}$, then $\|u^{\tau_0}\|_{H_0^1(\Omega)}\leq M$ for all $\tau_0<\infty$, and
\\
\item if $p>\frac{N}{2}$, then $\|u\|_{L^\infty(\Omega)}\leq M$.
\end{enumerate} 
\end{proposition}

\begin{remark}\label{thetaremark1}
It is worth noticing the following facts about Proposition \ref{estimatethm1}:
\begin{itemize}
\item Observe that the inequality $\theta\leq a(q)$ is equivalent to $\sigma(\theta)\geq 1$, and both imply that 
\[\theta<\left(\frac{N-1}{N}q-1\right)\frac{N}{N-2}.\]
One can check using \eqref{keyparameters} that this last inequality is equivalent to $r(\theta)>1$. In conclusion, $r>1$ in Proposition \ref{estimatethm1}.

\item Concerning the hypothesis in item (1), is is easy to see, using that $r$ is decreasing, $r(a(q))=(2^*)'$ and $r(0)=\frac{N(q-1)}{q}$, that it is equivalent to $h\in L^p(\Omega)$ for some
\begin{equation}\label{equivrestr}
\max\{r(\alpha),(2^*)'\}\leq p\leq \frac{N(q-1)}{q}.
\end{equation}
If $\alpha>0$, then \eqref{equivrestr} is obviously weaker than  $p=\frac{N(q-1)}{q}$, which is imposed in \cite[Theorem 5.8]{GMP}. Furthermore, if $\alpha=0$, then $p=\frac{N(q-1)}{q}$ and \eqref{equivrestr} are actually the same thing.

\item Notice also that, if $p=r(\theta)$, then $\tau=\sigma(\theta)$. In consequence, since the function $p\mapsto\tau(p)$ is strictly increasing, it holds that $\tau>\sigma(0)$ if $p>r(0)$.
\end{itemize}
\end{remark}

\begin{proof}[Proof of Proposition \ref{estimatethm1}]
{\bf Proof of (1).} For $k>0$, let us take $G_k(u)^{2\sigma-1}\in H_0^1(\Omega)\cap L^\infty(\Omega)$ as test function, where $\sigma=\sigma(\theta)$ (this choice can be made since $\sigma\geq 1$, see Remark \ref{thetaremark1}). Thus we obtain
\begin{equation}\label{eq1}
\beta\int_\Omega u G_k(u)^{2\sigma-1}+(2\sigma-1)\int_\Omega G_k(u)^{2(\sigma-1)}|\nabla G_k(u)|^2=\int_\Omega \left(\mu(x)\frac{|\nabla G_k(u)|^q}{u^\alpha} + |h(x)|\right)G_k(u)^{2\sigma-1}.
\end{equation}
It is clear that
\begin{equation}
(2\sigma-1)\int_\Omega G_k(u)^{2(\sigma-1)}|\nabla G_k(u)|^2=\frac{2\sigma-1}{\sigma^2}\int_\Omega|\nabla G_k(u)^\sigma|^2.
\end{equation}

Let us now estimate the nonlinear term. We have that
\begin{align*}
\int_\Omega \frac{|\nabla G_k(u)|^q}{u^\alpha} G_k(u)^{2\sigma-1}&\leq \frac{1}{k^{\alpha-\theta}}\int_{\Omega} |\nabla G_k(u)|^q G_k(u)^{2\sigma-1-\theta}
\\
&\leq C\left(\int_\Omega |\nabla G_k(u)^\sigma|^2\right)^\frac{q}{2}\left(\int_\Omega G_k(u)^{\frac{2}{2-q}(2\sigma-1-\theta-q(\sigma-1))}\right)^{1-\frac{q}{2}}.
\end{align*}
One can easily check that
\[\frac{2}{2-q}(2\sigma-1-\theta-q(\sigma-1))=2^*\sigma.\]
Hence, the Sobolev's embeddings yield 
\[\int_\Omega |\nabla G_k(u)|^q G_k(u)^{2\sigma-1-\theta}\leq C\left(\int_\Omega |\nabla G_k(u)^\sigma|^2\right)^{\frac{q}{2}+\frac{2^*}{2}\left(1-\frac{q}{2}\right)}.\]

We now focus on the last term in \eqref{eq1}. 
\begin{align*}
\int_\Omega |h(x)|G_k(u)^{2\sigma-1}&= \int_{\{|h(x)|\leq\beta u\}}|h(x)|G_k(u)^{2\sigma-1}+\int_{\{|h(x)|>\beta u\}}|h(x)|G_k(u)^{2\sigma-1}
\\
& \leq \beta\int_\Omega u G_k(u)^{2\sigma-1}+\int_{\{|h(x)|>\beta k\}}|h(x)|G_k(u)^{2\sigma-1}.
\end{align*}
The reader can check that
\[(2\sigma-1)r'=2^*\sigma.\]
Thus, H\"older's and Sobolev's inequalities imply that
\[\int_\Omega |h(x)|G_k(u)^{2\sigma-1}\leq  \beta\int_\Omega u G_k(u)^{2\sigma-1} +C \left(\int_{\{|h(x)|\geq \beta k\}}|h(x)|^r\right)^\frac{1}{r}\left(\int_\Omega |\nabla G_k(u)^\sigma|^2\right)^\frac{2^*}{2r'}.\]
If we denote $Y_k=\|G_k(u)^\sigma\|_{H_0^1(\Omega)}$, we have proved so far that
\[C Y_k^2\leq Y_k^{q+2^* \left(1-\frac{q}{2}\right)}+\|h\chi_{\{|h(x)|\geq\beta k\}}\|_{L^r(\Omega)}Y_k^\frac{2^*}{r'},\]
or equivalently,
\begin{equation}\label{ykineq}
C Y_k^{2-\frac{2^*}{r'}}-Y_k^{q+2^* \left(1-\frac{q}{2}\right)-\frac{2^*}{r'}}\leq\|h\chi_{\{|h(x)|\geq\beta k\}}\|_{L^r(\Omega)}.
\end{equation}
Let us define the function $F:[0,+\infty)\to\re$ by
\[F(Y)=C Y^{2-\frac{2^*}{r'}}-Y^{q+2^* \left(1-\frac{q}{2}\right)-\frac{2^*}{r'}}\quad\forall Y\geq 0.\]
Since $q<2$, it is easy to see that 
\[0<2-\frac{2^*}{r'}<q+2^*\left(1-\frac{q}{2}\right)-\frac{2^*}{r'}.\]
This means that $F$ is a concave function, positive near zero, negative far from zero, and has a unique maximum $F^*>0$ with a corresponding unique maximizer $Z^*>0$. Let us now consider
\[k^*=\inf\{k> 0:\|h\chi_{\{|h(x)|\geq\beta k\}}\|_{L^r(\Omega)}<F^*\}.\]
Hence, for any $\delta>0$, the equation $F(Y)=\|h\chi_{\{|h(x)|\geq\beta (k^*+\delta)\}}\|_{L^r(\Omega)}$ has two roots $Z_1$ and $Z_2$ such that $Z_1<Z^*<Z_2$. By virtue of inequality \eqref{ykineq}, it holds that for every $k\geq k^*+\delta$, either $Y_k\leq Z_1$ or $Y_k\geq Z_2$. But the function $k\mapsto Y_k$ is continuous and tends to zero as $k$ tends to infinity. Hence, 
\[Y_{k^*+\delta}\leq Z_1<Z^*.\]
If we let now $\delta$ tend to zero, we obtain that 
\[Y_{k^*}=\|G_{k^*}(u)^\sigma\|_{H_0^1(\Omega)}\leq Z^*.\] 
Notice that
\[G_{k^*}(u)=u-k^*\geq 1\quad\text{in the set }\{u\geq k^*+1\}.\]
Therefore,
\begin{align*}
\int_\Omega|\nabla G_{k^*+1}(u)|^2&=\int_\Omega\chi_{\{u\geq k^*+1\}}|\nabla u|^2 \leq \int_\Omega\chi_{\{u\geq k^*+1\}}|\nabla u|^2 G_{k^*}(u)^{2(\sigma-1)}
\\
&\leq  \int_\Omega\chi_{\{u\geq k^*\}}|\nabla u|^2 G_{k^*}(u)^{2(\sigma-1)}=\frac{1}{\sigma^2}\int_\Omega |\nabla G_{k^*}(u)^\sigma|^2\leq C.
\end{align*}
Now we take $T_{k^*+1}(u)$ as test function in the weak formulation of \eqref{auxproblem} so we get
\begin{align*}
\int_\Omega |\nabla T_{k^*+1}(u)|^2&= \int_\Omega\left(\mu(x)\frac{|\nabla u|^q}{u^\alpha}+h(x)\right)T_{k^*+1}(u)
\\
&\leq C(k^*+1)^{1-\alpha}\int_\Omega|\nabla u|^q+(k^*+1)\int_\Omega |h(x)|
\\
&\leq C(k^*+1)\left(\int_\Omega|\nabla T_{k^*+1}(u)|^q+\int_\Omega|\nabla G_{k^*+1}(u)|^q+1\right)
\\
&\leq C(k^*+1)\left(\int_\Omega|\nabla T_{k^*+1}(u)|^q+1\right).
\end{align*}
Since $q<2$, this clearly implies that 
\[\int_\Omega |\nabla T_{k^*+1}(u)|^2\leq C.\]
Note that, in principle, this last constant depends on $k^*$, which may in turn depend on $\|h\|_{L^r(\Omega)}$. However, since $\|h\|_{L^r(\Omega)}\leq C$, the absolute continuity of the integral implies that $k^*\leq k_0$ for some $k_0>0$ independent of $\|h\|_{L^r(\Omega)}$.

\smallskip

Summarizing, 
\[\int_\Omega |\nabla u|^2\leq C,\]
which proves the first part of item (1). Moreover,
\begin{align*}
\int_\Omega|\nabla u^\sigma|^2&=\int_\Omega|\nabla G_{k^*}(u)^\sigma|^2+\int_\Omega|\nabla T_{k^*}(u)^\sigma|^2
\\
&\leq (Z^*)^2+\sigma^2\int_\Omega T_{k^*}(u)^{2(\sigma-1)}|\nabla T_{k^*}(u)|^2\leq (Z^*)^2+\sigma^2(k^*)^{2(\sigma-1)}\int_\Omega|\nabla u|^2\leq C.
\end{align*}
Thus, the proof of item (1) is concluded.

\smallskip

{\bf Proof of (2).} Arguing as above, we take $G_k(u)^{2\tau-1}$ as test function in the weak formulation of \eqref{auxproblem} for some $k>0$, so we obtain
\[\beta\int_\Omega u G_k(u)^{2\tau-1}+\frac{2\tau-1}{\tau^2}\int_\Omega |\nabla G_k(u)^\tau|^2=\int_\Omega \left(\mu(x)\frac{|\nabla G_k(u)|^q}{u^\alpha} + |h(x)|\right)G_k(u)^{2\tau-1}.\]
In order to estimate the nonlinear term, notice that
\[\frac{q}{2}+\frac{2-q}{2^*}+\frac{2-q}{N}=1.\]
Hence, we can use H\"older inequality with those three exponents, and we deduce that
\begin{align*}
\int_\Omega \frac{|\nabla G_k(u)|^q}{u^\alpha} G_k(u)^{2\tau-1}&\leq C\int_\Omega |\nabla G_k(u)|^q G_k(u)^{(\tau-1)q}G_k(u)^{(2-q)\tau}G_k(u)^{q-1}
\\
&\leq C\left(\int_\Omega|\nabla G_k(u)^\tau|^2\right)^\frac{q}{2}\left(\int_\Omega G_k(u)^{\tau 2^*}\right)^\frac{2-q}{2^*}\left(\int_\Omega G_k(u)^{2^*\sigma(0)}\right)^\frac{2-q}{N}
\\
&\leq C\|G_k(u)\|^{q-1}_{L^{2^*\sigma(0)}(\Omega)}\int_\Omega |\nabla G_k(u)^\tau|^2.
\end{align*}
Now choose $k_0>0$, independent of $u$, such that 
\[C\|G_k(u)\|^{q-1}_{L^{2^*\sigma(0)}(\Omega)}<\frac{2\tau-1}{\tau^2}\quad\forall k\geq k_0.\]
Notice that this is possible thanks to part (1) of the theorem and the absolute continuity of the integral. Then, removing the positive linear term which contains $\beta$ and using H\"older's inequality in the term with $h$, we derive
\[C\int_\Omega|\nabla G_k(u)^\tau|^2\leq \|h\|_{L^p(\Omega)}\left(\int_\Omega G_k(u)^{(2\tau-1)p'}\right)^\frac{1}{p'}\quad\forall k\geq k_0.\]
Since $(2\tau-1)p'=2^*\tau$, we conclude that
\[C\int_\Omega |\nabla G_k(u)^\tau|^2\leq \left(\int_\Omega |\nabla G_k(u)^\tau|^2\right)^\frac{2^*}{2p'}\quad\forall k\geq k_0.\]
Clearly, $\frac{2^*}{2p'}=\frac{2\tau-1}{2\tau}<1$, so we deduce that
\[\int_\Omega|\nabla G_{k_0}(u)^\tau|^2\leq C.\]
Finally, using that $u$ is bounded in $H_0^1(\Omega)$ (from item (1)), we deduce that
\begin{align*}
\int_\Omega|\nabla u^\tau|^2&=\int_\Omega|\nabla G_{k_0}(u)^\tau|^2+\int_\Omega|\nabla T_{k_0}(u)^\tau|^2
\\
&\leq C+\tau^2\int_\Omega T_{k_0}(u)^{2(\tau-1)}|\nabla T_{k_0}(u)|^2\leq C+\tau^2 k_0^{2(\tau-1)}\int_\Omega|\nabla u|^2\leq C.
\end{align*}
This proves part (2) of the theorem.

\smallskip

{\bf Proof of (3).} Since $\tau\to+\infty$ as $p\to \frac{N}{2}$, part (3) is a clear consequence of part (2).

\smallskip 

{\bf Proof of (4).} Let us take $G_k(u)$ as test function in the weak formulation of \eqref{auxproblem} for some $k>0$, so we obtain this time, removing the term with $\beta$,
\[\int_\Omega |\nabla G_k(u)|^2\leq C \int_\Omega |\nabla G_k(u)|^q G_k(u)+\int_\Omega |h(x)|G_k(u).\]
Since $\frac{2}{2-q}=\left(1-\frac{2}{N}\right)2^*+\frac{2}{N}2^*\sigma(0)$, we deduce that
\begin{align*}
\int_\Omega |\nabla G_k(u)|^q G_k(u)&\leq\left(\int_\Omega |\nabla G_k(u)|^2\right)^\frac{q}{2}\left(\int_\Omega G_k(u)^\frac{2}{2-q}\right)^{1-\frac{q}{2}}
\\
&\leq \left(\int_\Omega |\nabla G_k(u)|^2\right)^\frac{q}{2}\left(\int_\Omega G_k(u)^{2^*}\right)^\frac{2-q}{2^*}\|G_k(u)\|_{L^{2^*\sigma(0)}(\Omega)}^{q-1}
\\
&\leq C\|G_k(u)\|_{L^{2^*\sigma(0)}(\Omega)}^{q-1}\int_\Omega |\nabla G_k(u)|^2.
\end{align*}
Next, as in part (2), we take $k\geq k_0$, with $k_0$ independent of $u$, so that $\|G_k(u)\|_{L^{2^*\sigma(0)}(\Omega)}^{q-1}$ is small enough. Then, 
\[C\int_\Omega|\nabla G_k(u)|^2\leq \int_\Omega |h(x)|G_k(u).\]
We conclude by using the Stampacchia's method in a direct way.
\end{proof}

The following result is analogous to Proposition \ref{estimatethm1}, but is is valid for a lower range for $q$. The proof is similar to the one above, but still there are relevant differences so it is included for the convenience of the reader.

\begin{proposition}\label{estimatethm2}
Assume that $q,\alpha,\beta,\mu$ satisfy \eqref{parameters} and let $r,\sigma$ be defined as in \eqref{keyparameters}. Assume also that 
\begin{equation*}
\frac{N}{N-1}<q<1+\frac{2}{N},
\end{equation*}
and let us denote
\begin{equation}
b(q)=\left(\frac{N-1}{N}q-1\right)\frac{N}{N-2}\in (0,q-1).
\end{equation}
Then, for all $C>0$ and $p>1$, there exists $M>0$ such that, for any $h\in L^p(\Omega)$ with $\|h\|_{L^p(\Omega)}\leq~C$ and for any solution $u$ to problem \eqref{auxproblem}, the following holds:
\begin{enumerate}
\item If $p=r(\theta)$ for some $\theta\in\left[0,\min\left\{\alpha,b(q)\right\}\right]\setminus\{b(q)\}$, then $\|u(1+u)^{\sigma(\theta)-1}\|_{H_0^1(\Omega)}\leq M$;
\\
\item if $r(0)<p\leq(2^*)'$, then $\|u(1+u)^{\tau-1}\|_{H_0^1(\Omega)}\leq M$, where $\tau=\frac{p(N-2)}{2(N-2p)}$;
\\
\item if $(2^*)'<p<\frac{N}{2}$, then $\|u^\tau\|_{H_0^1(\Omega)}\leq M$, where $\tau=\frac{p(N-2)}{2(N-2p)}$;
\\
\item if $p=\frac{N}{2}$, then $\|u^{\tau_0}\|_{H_0^1(\Omega)}\leq M$ for all $\tau_0<\infty$, and
\\
\item if $p>\frac{N}{2}$, then $\|u\|_{L^\infty(\Omega)}\leq M$.
\end{enumerate} 
\end{proposition}

\begin{remark}
Concerning the hypothesis in item (1) in the previous result, is is easy to see, using that $r$ is decreasing, $r(b(q))=1$ and $r(0)=\frac{N(q-1)}{q}$, that it is equivalent to $h\in L^p(\Omega)$ for some
\[p\in\left[\max\{r(\alpha),1\},\frac{N(q-1)}{q}\right]\setminus\{1\}.\]
This assumption is obviously weaker than  $p=\frac{N(q-1)}{q}$, which is imposed in \cite[Theorem 5.8]{GMP}. Actually, if $\alpha\geq b(q)$, then it is enough to impose that $p>1$. Notice also that, if $p=r(\theta)$, then $\tau=\sigma(\theta)$. In consequence, since the function $p\mapsto\tau(p)$ is increasing, it holds that $\tau>\sigma(0)$ if $p>r(0)$.
\end{remark}

\begin{proof}[Proof of Proposition \ref{estimatethm2}]
{\bf Proof of (1).} Let us consider the following functions defined for every $t\geq 0$:
\begin{align}
\phi(t)&=\frac{1}{(\zeta+t)^{1-\sigma(\theta)}}\left(\frac{t}{\zeta+t}\right)^\frac{1}{2},
\\
\Phi_1(t)&=\int_0^t\phi(s)ds,
\\
\Phi_2(t)&=\int_0^t\phi(s)^2 ds,
\end{align}
where $\zeta>0$ will be fixed later.

First of all observe that
\[\nabla v\nabla\Phi_2(v)=|\nabla\Phi_1(v)|^2\]
for any $v\in H_0^1(\Omega)$. Moreover, using that $\frac{2(1-\theta)}{2-q}-2(1-\sigma(\theta))=2^*\sigma(\theta)$ and $2\sigma(\theta)-1=\frac{2^*\sigma(\theta)}{r(\theta)'}$, it can be proved respectively that
\begin{equation}\label{ineqphi1}
\left(t^{-\theta}\phi(t)^{-q}\Phi_2(t)\right)^{\frac{2}{2-q}}\leq C\left(\Phi_1(t)^{2^*}+\zeta^{2^*\sigma(\theta)}\right)\quad\forall t\geq 0.
\end{equation}
and 
\begin{equation}\label{ineqphi2}
\Phi_2(t)\leq C\Phi_1(t)^\frac{2^*}{r(\theta)'}\quad\forall t\geq 0.
\end{equation}

For $k>0$, let us take $\Phi_2(G_k(u))\in H_0^1(\Omega)\cap L^\infty(\Omega)$ as test function in the weak formulation of \eqref{auxproblem}, so that we obtain
\begin{equation}\label{eq2}
\beta\int_\Omega u \Phi_2(G_k(u))+\int_\Omega |\nabla \Phi_1(G_k(u))|^2=\int_\Omega\left(\mu(x)\frac{|\nabla G_k(u)|^q}{u^{\alpha}}+ |h(x)|\right)\Phi_2(G_k(u)).
\end{equation}

Let us now estimate the nonlinear term. Thanks to \eqref{ineqphi1} we derive that
\begin{align*}
\int_\Omega \mu(x)&\frac{|\nabla G_k(u)|^q}{u^\alpha} \Phi_2(G_k(u))\leq \frac{\|\mu\|_{L^\infty(\Omega)}}{k^{\alpha-\theta}}\int_{\{u\geq k\}} |\nabla \Phi_1(G_k(u))|^q \frac{\Phi_2(G_k(u))}{G_k(u)^\theta \phi(G_k(u))^q}
\\
\leq&C\left(\int_\Omega |\nabla \Phi_1(G_k(u))|^2\right)^\frac{q}{2}\left(\int_{\{u\geq k\}} \left(\frac{\Phi_2(G_k(u))}{G_k(u)^\theta \phi(G_k(u))^q}\right)^{\frac{2}{2-q}}\right)^{1-\frac{q}{2}}
\\
\leq&C\left(\int_\Omega |\nabla \Phi_1(G_k(u))|^2\right)^\frac{q}{2}\left(\int_\Omega\left(\Phi_1(G_k(u))^{2^*}+\zeta^{2^*\sigma}\right)\right)^{1-\frac{q}{2}}
\\
\leq&C \left(\int_\Omega |\nabla \Phi_1(G_k(u))|^2\right)^\frac{q}{2}\left(\left(\int_\Omega |\nabla \Phi_1(G_k(u))|^2\right)^{\frac{2^*}{2}\left(1-\frac{q}{2}\right)}+\zeta^{2^*\sigma\left(1-\frac{q}{2}\right)}\right).
\end{align*}

We now focus on the last term in \eqref{eq2}. Using \eqref{ineqphi2} we deduce that
\begin{align*}
\int_\Omega |h(x)|\Phi_2(G_k(u))&= \int_{\{|h(x)|\leq\beta u\}}|h(x)|\Phi_2(G_k(u))+\int_{\{|h(x)|>\beta u\}}|h(x)|\Phi_2(G_k(u))
\\
& \leq \beta\int_\Omega u \Phi_2(G_k(u))+C\int_{\{|h(x)|>\beta k\}}|h(x)|\Phi_1(G_k(u))^\frac{2^*}{r'}
\\
&\leq\beta\int_\Omega u\Phi_2(G_k(u))+C \left(\int_{\{|h(x)|\geq \beta k\}}|h(x)|^{r}\right)^\frac{1}{r}\left(\int_\Omega \Phi_1(G_k(u))^{2^*}\right)^\frac{1}{r'}
\\
&\leq\beta\int_\Omega u\Phi_2(G_k(u))+C \left(\int_{\{|h(x)|\geq \beta k\}}|h(x)|^{r}\right)^\frac{1}{r}\left(\int_\Omega|\nabla \Phi_1(G_k(u)|^2\right)^\frac{2^*}{2r'}.
\end{align*}

If we denote $Y_k=\|\Phi_1(G_k(u))\|_{H_0^1(\Omega)}$, we have proved so far that

\[Y_k^2\leq C Y_k^q\left(Y_k^{2^*\left(1-\frac{q}{2}\right)}+\zeta^{2^*\sigma\left(1-\frac{q}{2}\right)}\right)+C\|h\chi_{\{|h(x)|\geq\beta k\}}\|_{L^r(\Omega)}Y_k^{\frac{2^*}{r'}}.\]

Hence, using Young's inequality we obtain that
\[\frac{1}{2} Y_k^2\leq C Y_k^{q+2^* \left(1-\frac{q}{2}\right)}+ C\zeta^{2^*\sigma}+C\|h\chi_{\{|h(x)|\geq\beta k\}}\|^{\frac{2r'}{2r'-2^*}}_{L^r(\Omega)},\]
or equivalently,
\begin{equation}\label{ykineq2}
C_1 Y_k^2-C_2 Y_k^{q+2^* \left(1-\frac{q}{2}\right)}\leq\zeta^{2^*\sigma}+\|h\chi_{\{|h(x)|\geq\beta k\}}\|^{\frac{2r'}{2r'-2^*}}_{L^r(\Omega)},
\end{equation}
for some $C_1, C_2>0$ independent of $k$ and $\zeta$.

Let us define the function $F:[0,+\infty)\to\re$ by
\[F(Y)=C_1 Y^2-C_2 Y^{q+2^* \left(1-\frac{q}{2}\right)}\quad\forall Y\geq 0.\]
Since $q<2$, it easy to see that 
\[2<q+2^*\left(1-\frac{q}{2}\right).\]
This means that $F$ is a concave function, positive near zero, negative far from zero, and has a unique maximum $F^*>0$ with a corresponding unique maximizer $Z^*>0$.          

We now choose $\zeta=\min\left\{1,\left(\frac{F^*}{2}\right)^{\frac{1}{2^*\sigma}}\right\}$. Thus, 
\[\max_{Y\geq 0}(F(Y)-\zeta^{2^*\sigma})=F^*-\zeta^{2^*\sigma}\geq\frac{F^*}{2}>0.\]
Let us now consider
\[k^*=\inf\left\{k> 0:\|h\chi_{\{|h(x)|\geq\beta k\}}\|^{\frac{2r'}{2r'-2^*}}_{L^r(\Omega)}<F^*-\zeta^{2^*\sigma}\right\}.\]
Hence, for any $\delta>0$, the equation $F(Y)=\zeta^{2^*\sigma}+\|h\chi_{\{h(x)\geq\beta (k^*+\delta)\}}\|^{\frac{2r'}{2r'-2^*}}_{L^r(\Omega)}$ has two roots $Z_1$ and $Z_2$ such that $Z_1<Z^*<Z_2$. By virtue of inequality \eqref{ykineq2}, it holds that for every $k\geq k^*+\delta$, either $Y_k\leq Z_1$ or $Y_k\geq Z_2$. But the function $k\mapsto Y_k$ is continuous and tends to zero as $k$ tends to infinity. Hence, 
\[Y_{k^*+\delta}\leq Z_1<Z^*.\]
If we let now $\delta$ tend to zero, we obtain that 
\[Y_{k^*}=\|\Phi_1(G_{k^*}(u))\|_{H_0^1(\Omega)}\leq Z^*.\] 
Notice that 
\begin{align*}
\|\Phi_1(G_k(u))\|^2_{H_0^1(\Omega)}&=\int_\Omega\frac{|\nabla G_k(u)|^2 G_k(u)}{(\zeta+G_k(u))^{2(1-\sigma)+1}}
\geq \int_\Omega\frac{|\nabla G_k(u)|^2 G_k(u)}{(1+G_k(u))^{2(1-\sigma)+1}}
\\
&\geq \int_\Omega\frac{|\nabla u|^2(u-k)}{(1+u-k)^{2(1-\sigma)+1}}\chi_{\{u\geq k+1\}}
\\
&\geq \frac{1}{2}\int_\Omega\frac{|\nabla u|^2}{(1+u-k)^{2(1-\sigma)}}\chi_{\{u\geq k+1\}}
\\
&\geq \frac{1}{2^{2(1-\sigma)+1}}\int_\Omega\frac{|\nabla G_{k+1}(u)|^2}{(G_{k+1}(u)+1)^{2(1-\sigma)}}.
\end{align*}
Hence, we have that
\begin{equation}\label{Gkestimate}
\int_\Omega\frac{|\nabla G_k(u)|^2}{(G_k(u)+1)^{2(1-\sigma)}}\leq C\quad\forall k\geq k^*+1.
\end{equation}
Fix $k\geq k^*+1$ independent of $\|h\|_{L^r(\Omega)}$. Note again that this can be done since $\|h\|_{L^r(\Omega)}\leq C$, so $\|h\chi_{\{|h(x)|\geq\beta k\}}\|_{L^r(\Omega)}\to 0$ uniformly in $\|h\|_{L^r(\Omega)}$ as $k\to\infty$. Then, estimate \eqref{Gkestimate} implies that
\begin{align*}
\left\|\frac{u}{(1+u)^{1-\sigma}}\right\|^2_{H_0^1(\Omega)}&=\int_\Omega\frac{|\nabla u|^2}{(1+u)^{2(1-\sigma)}}\left(\frac{1+\sigma u}{1+u}\right)^2
\\
&=\int_\Omega\frac{|\nabla G_k(u)|^2}{(1+u)^{2(1-\sigma)}}\left(\frac{1+\sigma u}{1+u}\right)^2+\int_\Omega\frac{|\nabla T_k(u)|^2}{(1+u)^{2(1-\sigma)}}\left(\frac{1+\sigma u}{1+u}\right)^2
\\
&\leq C+\int_\Omega\frac{|\nabla T_k(u)|^2}{(1+u)^{2(1-\sigma)}}\left(\frac{1+\sigma u}{1+u}\right)^2.
\end{align*}

We claim now that
\begin{equation}\label{Tkestimate}
\int_\Omega\frac{|\nabla T_k(u)|^2}{(1+u)^{2(1-\sigma)}}\left(\frac{1+\sigma u}{1+u}\right)^2\leq C.
\end{equation}
Indeed, let us define the real functions for all $s\geq 0$:
\begin{align*}
z(s)&=\frac{1}{(1+s)^{2(1-\sigma)}}\left(\frac{1+\sigma s}{1+s}\right)^2,
\\
y(s)&=e^{-\frac{s^2}{2}}\int_0^s e^\frac{t^2}{2}z(t)dt.
\end{align*}
It is easy to see that
\[y'(s)+sy(s)=z(s)\quad\forall s\geq 0,\]
and also that
\[y(s)\leq Cz(s)\quad\forall s\geq 0, \text{ for some }C>0.\]
Now we take $T_k(u)y(u)$ as test function in the weak formulation of \eqref{auxproblem} and get
\begin{equation}\label{Tkequation}
\int_\Omega y(u)|\nabla T_k(u)|^2+\int_\Omega T_k(u)y'(u)|\nabla u|^2=\int_\Omega \left(\mu(x)\frac{|\nabla u|^q}{u^\alpha}+h(x)\right)T_k(u)y(u).
\end{equation}
Concerning the right hand side of \eqref{Tkequation}, observe that
\begin{equation}\label{equ1}
\int_\Omega y(u)|\nabla T_k(u)|^2+\int_\Omega T_k(u)y'(u)|\nabla u|^2 =\int_\Omega |\nabla T_k(u)|^2 z(u)+k\int_\Omega y'(u)|\nabla G_k(u)|^2,
\end{equation}
where
\begin{equation}\label{equ2}
-k\int_\Omega y'(u)|\nabla G_k(u)|^2 \leq \int_\Omega \frac{k y(u)}{u}|\nabla G_k(u)|^2\leq C\int_\Omega z(u)|\nabla G_k(u)|^2\leq C.
\end{equation}
Gathering \eqref{Tkequation}, \eqref{equ1} and \eqref{equ2} together we deduce that
\[\int_\Omega z(u)|\nabla T_k(u)|^2\leq C \left(\int_\Omega y(u)|\nabla u|^q +1\right)\leq C\left(\int_\Omega z(u)|\nabla T_k(u)|^q +\int_\Omega z(u)|\nabla G_k(u)|^q +1\right).\]
We finally arrive at \eqref{Tkestimate} by using Young's inequality, by the fact that $z$ is a bounded function, and also by virtue of \eqref{Gkestimate}.

\smallskip

Thus, we have proved item (1).

{\bf Proof of (2).} 
Let us consider the following functions defined on $[0,+\infty)$:
\begin{align*}
\psi(t)&=\frac{1}{(1+t)^{1-\tau}}\left(\frac{t}{1+t}\right)^{\frac{1}{2}}\quad\forall t\geq 0,
\\
\Psi_1(t)&=\int_0^t \psi(s)ds\quad\forall t\geq 0,
\\
\Psi_2(t)&=\int_0^t\psi(s)^2ds\quad\forall t\geq 0.
\end{align*}
It can be easily proved that
\begin{equation}\label{psiineq1}
\Psi_2(t)\leq C\psi(t)\Psi_1(t)\quad\forall t\geq 0.
\end{equation}
Moreover, since $2\tau-1=\frac{2^*\tau}{p'}$, it is easy to prove that 
\begin{equation}\label{psiineq2}
\Psi_2(t)\leq C\Psi_1(t)^{\frac{2^*}{p'}}\quad\forall t\geq 0.
\end{equation}

For some $k>0$ we take $\Psi_2(G_k(u))$ as test function in the weak formulation of \eqref{auxproblem}, so that we obtain
\begin{equation}\label{ineq1item2}
\int_\Omega|\nabla\Psi_1(G_k(u))|^2\leq C\int_\Omega\left(\frac{|\nabla G_k(u)|^q}{u^\alpha}+|h(x)|\right)\Psi_2(G_k(u)).
\end{equation}

Concerning the singular term, using \eqref{psiineq1} and H\"older's and Sobolev's inequalities, we deduce that
\begin{align}\label{ineq2item2}
\nonumber\int_\Omega &\frac{|\nabla G_k(u)|^q}{u^\alpha}\Psi_2(G_k(u))\leq C\int_\Omega\frac{|\nabla G_k(u)|^{q-1}}{u^\theta}|\nabla G_k(u)|\psi(G_k(u))\Psi_1(G_k(u))
\\
&\leq C\left(\int_\Omega\frac{|\nabla G_k(u)|^{N(q-1)}}{u^{N\theta}}\right)^\frac{1}{N}\left(\int_\Omega|\nabla G_k(u)|^2\psi(G_k(u))^2\right)^{\frac{1}{2}}\left(\int_\Omega\Psi_1(G_k(u))^{2^*}\right)^{\frac{1}{2^*}}
\\
\nonumber &\leq C\left(\int_\Omega\frac{|\nabla G_k(u)|^{N(q-1)}}{u^{N\theta}}\right)^\frac{1}{N}\int_\Omega|\nabla\Psi_1(G_k(u))|^2.
\end{align}

Now we claim that 
\[\int_\Omega\frac{|\nabla G_k(u)|^{N(q-1)}}{u^{N\theta}}\leq C\]
for some $k>0$ large enough. Indeed, since $q<1+\frac{2}{N}$, we have that
\[\left(\frac{2}{N(q-1)}\right)'[(1-\sigma)N(q-1)-N\theta]=2^*\sigma,\]
and thus, for any $k>0$,
\begin{align*}
\int_\Omega\frac{|\nabla G_k(u)|^{N(q-1)}}{u^{N\theta}}&\leq\int_\Omega\frac{|\nabla G_k(u)|^{N(q-1)}}{(1+G_k(u))^{N\theta}}
\\
&\leq C\int_\Omega\left|\nabla\frac{G_k(u)}{(1+G_k(u))^{1-\sigma}}\right|^{N(q-1)}(1+G_k(u))^{(1-\sigma)N(q-1)-N\theta}
\\
&\leq C\left(\int_\Omega\left|\nabla\frac{G_k(u)}{(1+G_k(u))^{1-\sigma}}\right|^2\right)^{\frac{N(q-1)}{2}}\left(\int_\Omega (1+G_k(u))^{2^*\sigma}\right)^{1-\frac{N(q-1)}{2}}.
\end{align*}
Therefore, by item (1),
\[\int_\Omega\frac{|\nabla G_k(u)|^{N(q-1)}}{u^{N\theta}}\leq C\quad\forall k\geq k^*+1,\]
and the proof of the claim is done. As a consequence, it can be shown that the limit
\[\lim_{k\to\infty}\left(\int_\Omega\frac{|\nabla G_k(u)|^{N(q-1)}}{u^{N\theta}}\right)^\frac{1}{N}=0,\]
is uniform in $u$. Hence, from \eqref{ineq2item2} we deduce that there exists $k_0>0$ independent of $u$ such that
\[\int_\Omega \frac{|\nabla G_k(u)|^q}{u^\alpha}\Psi_2(G_k(u))\leq C\int_\Omega|\nabla\Psi_1(G_k(u))|^2\quad\forall k\geq k_0.\]
Then, we derive from \eqref{ineq1item2} that
\[\int_\Omega|\nabla\Psi_1(G_k(u))|^2\leq C\int_\Omega|h(x)|\Psi_2(G_k(u))\quad\forall k\geq k_0.\]
By virtue of \eqref{psiineq2} we immediately derive the estimate
\[\int_\Omega|\nabla\Psi_1(G_{k_0}(u))|^2\leq C.\]

We conclude the proof of part (2) of the proposition similarly as part (1).

\smallskip

{\bf Proof of (3).} It follows the same steps of part (2), but considering this time 
\[\psi(t)=t^{\tau-1}\quad\forall t\geq 0.\]
Notice that this choice is valid as $\tau>1$ whenever $p>(2^*)'$. 

\smallskip

{\bf Proof of (4).} Since $\tau\to+\infty$ as $p\to \frac{N}{2}$, part (4) is a clear consequence of part (3).

\smallskip

{\bf Proof of (5).} It follows again the line of part (2) but with $\psi\equiv 1$, so that $\Psi_1(t)=\Psi_2(t)=t$ for all $t\geq 0$. The proof finishes by using the well-known Stampacchia's Lemma, as in Proposition \ref{estimatethm1}.
\end{proof}

We prove now a result analogous to Propositions \ref{estimatethm1} and \ref{estimatethm2} for $q$ small.

\begin{proposition}\label{estimateprop3}
Assume that $q,\alpha,\beta,\mu$ satisfy \eqref{parameters} and let $r,\sigma$ be defined as in \eqref{keyparameters}. Assume also that 
\begin{equation*}
q<\frac{N}{N-1}.
\end{equation*}
Then, for all $C>0$ and $p\geq 1$, there exist $M, k>0$ such that, for any $h\in L^p(\Omega)$ with $\|h\|_{L^p(\Omega)}\leq C$ and for any solution $u$ to problem \eqref{auxproblem}, the following holds:
\begin{enumerate}
\item If $p=1$, then $\|u\|_{\mathcal{M}^{\frac{N}{N-2}}(\Omega)}+\||\nabla u|\|_{\mathcal{M}^{\frac{N}{N-1}}(\Omega)}\leq M$;
\\
\item if $1<p\leq(2^*)'$, then $\|u(1+u)^{\tau-1}\|_{H_0^1(\Omega)}\leq M$, where $\tau=\frac{p(N-2)}{2(N-2p)}$;
\\
\item if $(2^*)'<p<\frac{N}{2}$, then $\|u^\tau\|_{H_0^1(\Omega)}\leq M$, where $\tau=\frac{p(N-2)}{2(N-2p)}$;
\\
\item if $p=\frac{N}{2}$, then $\|u^{\tau_0}\|_{H_0^1(\Omega)}\leq M$ for all $\tau_0<\infty$, and
\\
\item if $p>\frac{N}{2}$, then $\|u\|_{L^\infty(\Omega)}\leq M$.
\end{enumerate} 
\end{proposition}

\begin{proof}
{\bf Proof of (1).} For $j,k>0$, let us take $T_j(G_k(u))\in H_0^1(\Omega)\cap L^\infty(\Omega)$ as test function in the weak formulation of \eqref{auxproblem}. Thus we obtain
\begin{equation}\label{TkEquation}
\beta\int_\Omega uT_j(G_k(u))+\int_\Omega\nabla u\nabla T_j(G_k(u))=\int_\Omega\left(\mu(x)\frac{|\nabla G_k(u)|^q}{u^\alpha}+ |h(x)|\right)T_j(G_k(u)).
\end{equation}

On the one hand, it is clear that
\begin{equation*}
\int_\Omega \nabla u\nabla T_j(G_k(u))=\int_\Omega\left|\nabla  T_j(G_k(u))\right|^2.
\end{equation*}
On the other hand, concerning the right hand side of \eqref{TkEquation}, we obtain that
\begin{align*}
\int_\Omega \left(\mu(x)\frac{|\nabla G_k(u)|^q}{u^\alpha}+|h(x)|\right)T_j(G_k(u))&\leq j C\left(\int_\Omega|\nabla G_k(u)|^q+\int_{\{|h(x)|\geq \beta k\}}|h(x)|\right)
\\
&+\beta\int_\Omega u T_j(G_k(u)).
\end{align*}

In sum, we deduce that
\[\int_\Omega|\nabla  T_j(G_k(u))|^2\leq j C\left(\int_\Omega |\nabla G_k(u)|^q+\int_{\{|h(x)|\geq \beta k\}}|h(x)|\right).\]

Then, we apply \cite[Lemma 4.2]{BBGGPV}, so that we deduce that
\begin{equation*}
\|\nabla G_k(u)\|_{\mathcal{M}^{\frac{N}{N-1}}(\Omega)}\leq C\left(\int_\Omega|\nabla G_k(u)|^q+\int_{\{|h(x)|\geq \beta k\}}|h(x)|\right).
\end{equation*}

Since $q<\frac{N}{N-1}$, we have the immersions
\[\mathcal{M}^{\frac{N}{N-1}}(\Omega)\subset L^{\frac{N}{N-1}}(\Omega)\subset L^q(\Omega).\]
Therefore, 
\[C\|\nabla G_k(u)\|_{L^q(\Omega)}\leq \int_\Omega|\nabla G_k(u)|^q+\int_{\{|h(x)|\geq \beta k\}}|h(x)|.\]

We now consider the function $F:[0,\infty)\to \mathbb{R}$ defined as
\[F(Y)=CY-Y^q\quad\forall Y\geq 0,\]
and we denote 
\[Y_k=\|\nabla G_k(u)\|_{L^q(\Omega)}.\]
Thus we have proved that
\[F(Y_k)\leq \|h\chi_{\{|h(x)|\geq \beta k\}}\|_{L^1(\Omega)}.\]
The proof of this part concludes as in the previous proposition.

{\bf Proof of (2-5).} The proofs of the rest of the items follow the steps of the corresponding ones from Proposition \ref{estimatethm2}. The only part which is not completely straightforward is the proof of the estimate
\[\int_\Omega\frac{|\nabla G_k(u)|^{N(q-1)}}{u^{N\theta}}\leq C\quad\forall k\geq k_0.\]
However, since $q<\frac{N}{N-1}$, then $N(q-1)<\frac{N}{N-1}$, so we deduce that
\[\int_\Omega\frac{|\nabla G_k(u)|^{N(q-1)}}{u^{N\theta}}\leq C\int_\Omega|\nabla G_k(u)|^{N(q-1)}\leq C\left(\int_\Omega|\nabla G_k(u)|^{\frac{N}{N-1}}\right)^{(N-1)(q-1)}.\]
Therefore, the estimate holds by virtue of part (1).
\end{proof}

The same arguments of the proof of Proposition \ref{gmpbootstrap1} (but using Propositions \ref{estimatethm1}, \ref{estimatethm2} and \ref{estimateprop3} instead of the results in \cite{GMP}) are valid also for proving the main result of this subsection.

\begin{proposition}\label{mainestimate}
Assume that \eqref{H1} holds. If $q>\frac{N}{N-1}$, suppose also that \eqref{mrcondition} is satisfied. Furthermore, if $q\geq 1+\frac{2}{N}$, assume that \eqref{qlargecondition} is satisfied too. Then, for every $\lambda_0>0$, there exists $C>0$ such that
\[\|u\|_{L^\infty(\Omega)}\leq C\]
for every solution $u$ to \eqref{problem} with $\lambda>\lambda_0$.
\end{proposition}

\begin{remark}
Notice that, in principle, one can not apply  Propositions \ref{estimatethm2} nor \ref{estimateprop3} to prove Proposition \ref{mainestimate} in the case $q=\frac{N}{N-1}$. However, for $\varepsilon>0$ small, we have that $\frac{N}{N-1}+\varepsilon<1+\frac{2}{N}$ and 
\[\frac{|\nabla u|^\frac{N}{N-1}}{u^\alpha}\chi_{\{u\geq k\}}\leq \frac{|\nabla u|^{\frac{N}{N-1}+\varepsilon}}{u^\alpha}\chi_{\{u\geq k\}}+C_\varepsilon\]
for any $k>0$ and any solution $u$ to \eqref{problem}. Hence, the conclusions of Proposition \ref{estimatethm2} hold for $q=\frac{N}{N-1}+\varepsilon$.
\end{remark}

\smallskip

\subsection{Proof of the main result and consequences}\text{}

\smallskip

We prove now the main result of the paper.
\begin{proof}[Proof of Theorem \ref{maintheorem}]
Since there is a solution $u_0$ to $(P_0)$, then Proposition \ref{prop:continuum} (see also Remark \ref{rem:continuum}) implies that there exists an unbounded connected set $\Sigma^+$ such that
\[(0,u_0)\in\Sigma^+\subset ([0,+\infty)\times L^\infty(\Omega))\cap \Sigma,\]
where
\[\Sigma=\{(\lambda,u)\in \re\times L^\infty(\Omega):u \text{ is a solution to }\eqref{problem}\}.\] 
We claim that $\Sigma^+$ bifurcates from infinity to the right of the axis $\lambda=0$. Indeed, since \eqref{problem} does not have any solution for $\lambda\geq\lambda_1$ (see Remark \ref{lambda1remark1}), then $\Sigma^+\subset ([0,\lambda_1)\times L^\infty(\Omega))\cap \Sigma$. Therefore, since $\Sigma^+$ is unbounded, then its projection onto $L^\infty(\Omega)$ is unbounded. Now, Proposition \ref{mainestimate} implies that $\Sigma^+\cap((\lambda_0,\lambda_1)\times L^\infty(\Omega))$ is bounded for all $\lambda_0\in (0,\lambda_1)$. That is to say, $\Sigma^+\cap((0,\lambda_0)\times L^\infty(\Omega))$ is unbounded for all $\lambda_0>0$, and our claim is true.

\smallskip

We have proved that there exists a sequence $\{(\lambda_n,u_n)\}\subset\Sigma^+$ such that $\lambda_n\to 0$ and $\|u_n\|_{L^\infty(\Omega)}\to +\infty$ as $n\to +\infty$. We will show now that this fact and the connection of $\Sigma^+$ are enough to proof multiplicity of solutions for all $\lambda>0$ small enough. Indeed, assume by contradiction that there exists another sequence $\{(\mu_n,v_n)\}\subset\Sigma^+$ such that $\mu_n\to 0$ as $n\to\infty$ and $(P_{\mu_n})$ admits no other solution but $v_n$ for all $n$. On the other hand, using that $(0,u_0)\in \Sigma^+$ and $\Sigma^+$ is connected, it is clear that $\Sigma^+\cap B_r((0,u_0)) \setminus\{(0,u_0)\}\not=\emptyset$ for all $r>0$, where $B_r((0,u_0))$ denotes the open ball in $\mathbb{R}\times L^\infty(\Omega)$ centered at $(0,u_0)$ with radius $r$. Hence, since $v_n$ is unique and $\mu_n\to 0$, we have that, for all $r>0$, there exists $n_r\in\mathbb{N}$ such that, if $n\geq n_r$, then $(\mu_n,v_n)\in \Sigma^+\cap B_r((0,u_0)) \setminus\{(0,u_0)\}$. In other words, $v_n\to u_0$ in $L^\infty(\Omega)$ as $n\to 0$. Let us now take a not relabeled subsequence $\{(\mu_n,v_n)\}$ such that $\mu_{n+1}<\lambda_n<\mu_n$ for all $n$. Let us also fix $\eta>\|u_0\|_{L^\infty(\Omega)}$, and take $n$ large enough so that $\max\{\|v_n\|_{L^\infty(\Omega)},\|v_{n+1}\|_{L^\infty(\Omega)}\}<\eta<\|u_n\|_{L^\infty(\Omega)}$. We claim that there exists $(\nu_n,w_n)\in \Sigma^+$ such that $\nu_n\in (\mu_{n+1},\mu_n)$ and $\|w_n\|_{L^\infty(\Omega)}=\eta$.

Indeed, let us consider the set 
\[A_{n,\eta}=\{(\lambda,u)\in\Sigma:\lambda\in(\mu_{n+1},\mu_n), \|u\|_{L^\infty(\Omega)}=\eta\}.\]
Arguing by contradiction, assume that $\Sigma^+\cap A_{n,\eta}=\emptyset$. Let us define also
\[B_{n,\eta}=\{(\lambda,u)\in\Sigma: \lambda\in\{\mu_{n+1},\mu_n\}, \|u\|_{L^\infty(\Omega)}>\eta\}.\]
On the one hand, the uniqueness of $v_n$ and the fact that $\max\{\|v_n\|_{L^\infty(\Omega)},\|v_{n+1}\|_{L^\infty(\Omega)}\}<\eta$ imply that $\Sigma^+\cap B_{n,\eta}=\emptyset$. On the other hand, if we consider the set
\[U_{n,\eta}=\{(\lambda,u)\in\Sigma^+:\lambda\in(\mu_{n+1},\mu_n), \|u\|_{L^\infty(\Omega)}>\eta\},\]
then it is clear that $U_{n,\eta}$ is open in $\Sigma^+$, $(\lambda_n,u_n)\in U_{n,\eta}$ and $\partial U_{n,\eta}=A_{n,\eta}\cup B_{n,\eta}$. Hence, denoting $V_{n,\eta}=\Sigma^+\setminus\overline{U_{n,\eta}}$, we deduce that $V_{n,\eta}$ is also nonempty and open in $\Sigma^+$, $U_{n,\eta}\cap V_{n,\eta}=\emptyset$ and $\Sigma^+=U_{n,\eta}\cup V_{n,\eta}$. This contradicts that $\Sigma^+$ is connected.

Therefore, we have found a sequence $\{(\nu_n,w_n)\}\subset \Sigma^+$ such that $\nu_n\to 0$ as $n\to +\infty$ and $\|w_n\|_{L^\infty(\Omega)}=\eta$ for all $n$ large enough. In particular, $\{w_n\}$ is bounded in $L^\infty(\Omega)$. Then, we can argue as in the proof of Proposition \ref{exislambdaneg} in order to pass to the limit in $(P_{\nu_n})$. Thus, there exists $w\in H_0^1(\Omega)\cap L^\infty(\Omega)$ such that $w_n\rightharpoonup w$ weakly in $H_0^1(\Omega)$, $w_n\to w$ strongly in $L^\infty(\Omega)$ and $w$ is a solution to $(P_0)$. But $\|w\|_{L^\infty(\Omega)}=\eta>\|u_0\|_{L^\infty(\Omega)}$. This is a contradiction, as $u_0$ is unique by virtue of Theorem \ref{comprinc1} and Remark \ref{lambda1remark1}. The proof in now concluded.
\end{proof}

We conclude the section by stating and proving three corollaries of Theorem \ref{maintheorem}. The first one provides multiplicity of solutions for $q$ small, but for any $\alpha\in [0,q-1)$.

\begin{corollary}\label{corollary1}
Assume that \eqref{H1} holds with $q\in (1,Q_N]\setminus\{2\}$, where $Q_N$ is defined in \eqref{Qn}. If $N\leq 5$, assume also that $q<1+\frac{2}{N}$. Then, the conclusions of Theorem \ref{maintheorem} hold true.
\end{corollary}

\begin{remark}
Observe that $Q_N>1+\frac{2}{N}$ for all $N\leq 5$, while $Q_N<1+\frac{2}{N}$ otherwise. That is why we need to introduce an additional restriction in Corollary \ref{corollary1} for low dimensions. We will make a more detailed study of the case $q\geq 1+\frac{2}{N}$ for dimensions $N=3,4,5$ in Corollary \ref{corollary3} below.
\end{remark}

\begin{proof}[Proof of Corollary \ref{corollary1}]
Consider the function $z:[0,q-1)\to\re$ given by \[z(s)=\frac{q-s}{N-q+1}-\frac{q-1-s}{q-2s}\quad\forall s\in[0,q-1).\] 
It can be proved that $z$ is increasing. Indeed,
\begin{align*}
&Nz'(s)=-\frac{1}{N-q+1}+\frac{2-q}{(q-2s)^2}
\\
&=\frac{4}{(N-q+1)(q-2s)^2}\left(s-\frac{q-\sqrt{(2-q)(N+1-q)}}{2}\right)\left(\frac{q+\sqrt{(2-q)(N+1-q)}}{2}-s\right).
\end{align*}
Using that $N\geq 3$ and $q<2$, it is straightforward to deduce that 
\[\frac{q-\sqrt{(2-q)(N+1-q)}}{2}<0\quad\text{and}\quad\frac{q+\sqrt{(2-q)(N+1-q)}}{2}>q-1,\]
which means that $z'(s)>0$ for all $s\in [0,q-1)$. Moreover, since $q\leq Q_N$, then $z(0)\geq 0$ (see Proposition \ref{gmpbootstrap1}). Thus, $z(\alpha)\geq 0$, or equivalently, condition \eqref{mrcondition} holds and Theorem \ref{maintheorem} can be applied.
\end{proof}

The second corollary gives multiplicity of solutions for a wider range of $q$ at the expense of taking $\alpha$ somehow close to $q-1$.
\begin{corollary}\label{corollary2}
Assume that \eqref{H1} holds with $q\in \left(1, 1+\frac{2}{N}\right)$. If $q>\frac{N}{N-1}$, suppose also that $\alpha\geq \left(\frac{N-1}{N}q-1\right)\frac{N}{N-2}$. Then, the conclusions of Theorem \ref{maintheorem} hold true.
\end{corollary}

\begin{proof}
One only has to notice that, if $q>\frac{N}{N-1}$ and $\alpha\geq\left(\frac{N-1}{N}q-1\right)\frac{N}{N-2}$, then $\frac{N(q-1-\alpha)}{q-2\alpha}\leq 1$, while $\frac{(q-\alpha)N}{N-q+1}>1$. That is to say, \eqref{mrcondition} holds and Theorem \ref{maintheorem} can be applied. 
\end{proof}

Finally, the last consequence of Theorem \ref{maintheorem} provides multiplicity of solutions for $q$ close to $2$, but in this case more restrictive conditions have to be imposed on $\alpha$, and even on $N$.

\begin{corollary}\label{corollary3}
Assume that \eqref{H1} holds with $q\in\left[1+\frac{2}{N},2\right)$. Suppose in addition that one of the following conditions is satisfied:
\begin{enumerate}
\item $N=3$,
\
\item $N=4$ and $\alpha\in \left[0,q-1-\frac{2-q}{3}\right]$, or
\
\item $N=5$, $q\in\left[\frac{7}{5},\frac{3}{2}\right]$ and $\alpha\in \left[0,q-1-\frac{5-2q}{7}\right]$.
\end{enumerate}
Then, the conclusions of Theorem \ref{maintheorem} hold true. Moreover, if $N=5$ and $q\in\left(\frac{3}{2},\frac{13}{8}\right]$, there exists $\alpha_q\in\left(0,q-1-\frac{5-2q}{7}\right]$ such that, if $\alpha\in\left[\alpha_q,q-1-\frac{5-2q}{7}\right]$, then the conclusions of Theorem \ref{maintheorem} hold true.
\end{corollary}

\begin{proof}
Concerning items (1-3), one only has to check that $q\in \left[1+\frac{2}{N},Q_N\right]$ and $\alpha\leq \frac{q(N+4)-2(N+1)}{N+2}$ for the corresponding value of $N$, so conditions \eqref{mrcondition} and \eqref{qlargecondition} are satisfied and Theorem \ref{maintheorem} applies. Regarding the last statement, note that $q-1-\frac{5-2q}{7}=\frac{q(N+4)-2(N+1)}{N+2}$ for $N=5$, so condition \eqref{qlargecondition} holds for all $\alpha\in\left[0,q-1-\frac{5-2q}{7}\right]$. Besides, observe that $q>\frac{3}{2}=Q_5$. Then, considering the function $z(s)$ defined in the proof of Corollary \ref{corollary1}, it is clear that $z(0)<0$. On the other hand, one can easily check that $z\left(q-1-\frac{5-2q}{7}\right)=0$ if $q=\frac{13}{8}$, while $z\left(q-1-\frac{5-2q}{7}\right)>0$ provided $q<\frac{13}{8}$. In the first case, we choose $\alpha_q=q-1-\frac{5-2q}{7}$. In the second one, by continuity of $z$, there exists $\alpha_q\in\left(0,q-1-\frac{5-2q}{7}\right)$ such that $z(\alpha_q)=0$. Since $z$ is increasing, we have that \eqref{mrcondition} holds for all $\alpha\in [\alpha_q,q-1)$. In conclusion, Theorem \ref{maintheorem} can be used for $\alpha\in \left[\alpha_q,q-1-\frac{5-2q}{7}\right]$, and the proof is finished.
\end{proof}

\begin{remark}
If $N\geq 6$ and $q\geq 1+\frac{2}{N}$, it is straightforward to see that $z\left(\frac{q(N+4)-2(N+1)}{N+2}\right)<0$. Thus, since $z$ is increasing, $z(\alpha)<0$ for all $\alpha\in \left[0,\frac{q(N+4)-2(N+1)}{N+2}\right]$. Therefore, Theorem \ref{maintheorem} does not yield any information in this case.
\end{remark}

\section{Uniqueness for $q-1<\alpha\leq 1$}\label{sec:sublinear} 

We will consider in this section problem \eqref{problem} under condition \eqref{H2}. Observe that if $0<u\in W_{\mbox{\tiny loc}}^{1,1}(\Omega)$ and $t>0$, then
\[\frac{|\nabla tu|^q}{(tu)^\alpha}=t^{q-\alpha}\frac{|\nabla u|^q}{u^\alpha}.\]
In this case, $\alpha>q-1$, so $q-\alpha<1$. That is to say, the lower order term has \emph{sublinear homogeneity}. 

\begin{remark}\label{lambda1remark2}
The conclusions of Remark \ref{lambda1remark1} are valid also under hypothesis \eqref{H2}.
\end{remark}

We will prove the existence of solution to \eqref{problem} after deriving certain a priori estimates on an approximate problem and passing eventually to the limit, in a way that such a limit will be the solution we look for. Thus, consider the following approximate problem:

\begin{equation}\label{approblem}
\begin{cases}
\dys-\Delta u_n=\lambda u_n+\mu(x)\frac{T_n(|\nabla u_n|^q)}{\left(|u_n|+\frac{1}{n}\right)^\alpha}+T_n(f(x))\quad &\text{ in }\Omega,
\\
u_n=0\quad &\text{ on }\partial\Omega.
\end{cases}
\end{equation}

In the next lemma we show that problem \eqref{approblem} admits a solution.

\begin{lemma}\label{approxlemma}
Assume that \eqref{H2} holds. Then there exists a solution $u_n\in H_0^1(\Omega)\cap L^\infty(\Omega)$ to problem \eqref{approblem} for all $n\in\mathbb{N}$ and for all $\lambda<\lambda_1$. 
\end{lemma}

\begin{proof}
Fix $n\in\mathbb{N}$ and $\lambda<\lambda_1$. Then, the following linear problem has a solution $0<\overline{\psi}\in H_0^1(\Omega)\cap L^\infty(\Omega)$:
\begin{equation*}
\begin{cases}
\dys-\Delta u=\lambda u+n^{1+\alpha}\mu(x)+n\quad &\text{ in }\Omega,
\\
u=0\quad &\text{ on }\partial\Omega.
\end{cases}
\end{equation*}
Clearly, $\overline{\psi}$ is a supersolution to \eqref{approblem}. Moreover, $\underline{\psi}=0$ is a subsolution to \eqref{approblem}. Since $\underline{\psi}\leq\overline{\psi}$, then there exists a solution $u_n\in H_0^1(\Omega)\cap L^\infty(\Omega)$ to \eqref{approblem} (see \cite{BMP3}).
\end{proof}

We prove now the key estimates for proving the existence of solution to problem \eqref{problem}.

\begin{proposition}\label{modelestimates}
Assume that \eqref{H2} holds, and let $\lambda<\lambda_1$. Then there exist $\eta\in (0,1)$ and $C>0$ such that 
\[\|u_n\|_{H_0^1(\Omega)}+\|u_n\|_{C^{0,\eta}(\overline{\Omega})}\leq C\]
for every solution $u_n$ to \eqref{approblem} and for every $n$.
\end{proposition}

\begin{proof}

\noindent{\bf Step 1: $H_0^1$ estimate.}

Let us take $u_n$ as test function in the weak formulation of \eqref{approblem}. Then we obtain by using Poincar\'e's and H\"older's inequalities that
\begin{align*}
\int_\Omega|\nabla u_n|^2 &\leq\lambda\int_\Omega u_n^2+\|\mu\|_{L^\infty(\Omega)}\int_\Omega|\nabla u_n|^q u_n^{1-\alpha}+\int_\Omega f(x)u_n
\\
&\leq\frac{\lambda}{\lambda_1}\int_\Omega|\nabla u_n|^2+C\left(\int_\Omega |\nabla u_n|^2\right)^\frac{q}{2}\left(\int_\Omega u_n^{\frac{2(1-\alpha)}{2-q}}\right)^{1-\frac{q}{2}}+C\left(\int_\Omega u_n^{2^*}\right)^{\frac{1}{2^*}}.
\end{align*}
Now, since $\alpha>q-1$, then $\frac{2(1-\alpha)}{2-q}<2<2^*$. Hence, we can apply Sobolev's inequality to get that
\[\left(1-\frac{\lambda}{\lambda_1}\right)\int_\Omega |\nabla u_n|^2\leq C\left(\int_\Omega |\nabla u_n|^2\right)^\frac{q+1-\alpha}{2}+C\left(\int_\Omega|\nabla u_n|^2\right)^\frac{1}{2}.\]
Observe now that $\frac{q+1-\alpha}{2}<1.$ Therefore, we deduce that $\|u_n\|_{H_0^1(\Omega)}\leq C$.

\smallskip

\noindent{\bf Step 2: $L^\infty$ estimate.}

Assume now, in order to achieve a contradiction, that $\{\|u_n\|_{L^\infty(\Omega)}\}_{n\in\mathbb{N}}$ is unbounded, and choose a not relabeled divergent subsequence. Then, the function $v_n=\frac{u_n}{\|u_n\|_{L^\infty(\Omega)}}$ satisfies 

\begin{equation}\label{approblem2}
\begin{cases}
\dys-\Delta v_n=\lambda v_n+\frac{\mu(x)T_n(|\nabla u_n|^q)}{\|u_n\|_{L^\infty(\Omega)}\left(u_n+\frac{1}{n}\right)^{q-1+\alpha}}+\frac{f(x)}{\|u_n\|_{L^\infty(\Omega)}}\quad &\text{ in }\Omega,
\\
v_n>0\quad&\text{ in }\Omega,
\\
v_n=0\quad &\text{ on }\partial\Omega.
\end{cases}
\end{equation}
Notice that $\|v_n\|_{L^\infty(\Omega)}=1$ for all $n$, and also that
\begin{equation}\label{ineq1}
0\leq \frac{\mu(x)T_n(|\nabla u_n|^q)}{\|u_n\|_{L^\infty(\Omega)}\left(u_n+\frac{1}{n}\right)^{q-1+\alpha}}\leq\frac{\|\mu\|_{L^\infty(\Omega)}|\nabla v_n|^q}{\|u_n\|^\alpha_{L^\infty(\Omega)}v_n^{q-1+\alpha}}.
\end{equation}
Then, it is standard to prove that $\|v_n\|_{C^{0,\eta}(\overline{\Omega})}\leq C$ for all $n$ and for some $\eta\in (0,1)$ independent of $n$ following the arguments in \cite{LU} (see \cite[Appendix]{CLLM}). Hence, by Arzel\`a-Ascoli theorem, there exists $v\in C(\overline{\Omega})$ such that, up to a subsequence, $v_n\to v$ uniformly in $\overline{\Omega}$. Necessarily, $\|v\|_{L^\infty(\Omega)}=1$, so $v\not\equiv 0$. Moreover, by using the strong maximum principle conveniently, $v>0$ in $\Omega$. This last fact combined with the uniform convergence implies that, 
\[\forall\omega\subset\subset\Omega\,\,\exists c_\omega>0:\,\,v_n\geq c_\omega\text{ in }\omega.\] See the proof of \cite[Proposition 5.2]{CLLM} for more details.

Let now $\phi\in C_c^1(\Omega)$ be such that $\supp(\phi)\subset\omega$ for some open set $\omega\subset\subset\Omega$. Then, from \eqref{ineq1} we deduce that
\[\left|\int_\Omega\frac{\mu(x)T_n(|\nabla u_n|^q)\phi}{\|u_n\|_{L^\infty(\Omega)}\left(u_n+\frac{1}{n}\right)^{q-1+\alpha}}\right|\leq \frac{\|\mu \phi\|_{L^\infty(\Omega)}}{\|u_n\|^\alpha_{L^\infty(\Omega)}c_\omega^{q-1+\alpha}}\int_\omega|\nabla v_n|^q.\]
Using now that $\{v_n\}_{n\in\mathbb{N}}$ is bounded in $H_0^1(\Omega)$, we conclude that
\[\left|\int_\Omega\frac{\mu(x)T_n(|\nabla u_n|^q)\phi}{\|u_n\|_{L^\infty(\Omega)}\left(u_n+\frac{1}{n}\right)^{q-1+\alpha}}\right|\to 0\]
as $n\to\infty$.

Finally, we pass to the limit in \eqref{approblem2} and obtain that
\begin{equation*}
\begin{cases}
-\Delta v=\lambda v&\text{ in }\Omega,
\\
v>0&\text{ in }\Omega,
\\
v=0&\text{ on }\partial\Omega.
\end{cases}
\end{equation*}
This contradicts the fact that $\lambda<\lambda_1$.
\end{proof}

We are ready now to prove the main theorem of this section.

\begin{proof}[Proof of Theorem \ref{sublinearthm}]
Concerning the existence of solution, one has only to pass the limit in \eqref{approblem} using the a priori estimates in Proposition \ref{modelestimates}. The proof is similar to the one of Proposition~\ref{exislambdaneg}. The nonexistence of solution comes from Remark \ref{lambda1remark1}.

\smallskip

On the other hand, the uniqueness of solution is a direct consequence of Theorem \ref{comprinc2} and Remark \ref{lambda1remark1}.

\smallskip

Finally, similar arguments as in the proof of Step 2 in Proposition \ref{modelestimates} can be used to prove that $\lambda_1$ is the only possible bifurcation point from infinity. Actually, reasoning by contradiction and using that there is no solution to $(P_{\lambda_1})$, it is also standard to prove that $\lambda_1$ is, indeed, a bifurcation point from infinity.
\end{proof}

\section{Appendix: Existence of an unbounded continuum}\label{sec:appendix}

For every $w\in L^\infty(\Omega)$ and $\lambda\in\re$, let us consider the following problem:

\begin{equation}\label{continuumproblem}
\begin{cases}
\dys-\Delta u+u=\mu(x)\frac{|\nabla u|^q}{u^\alpha}+f(x)+(\lambda^++1)w^+\quad &\text{ in }\Omega,
\\
u>0 \quad &\text{ in }\Omega,
\\
u=0\quad &\text{ on }\partial\Omega.
\end{cases}
\end{equation}
If \eqref{H1} is satisfied, it is clear from Proposition \ref{exislambdaneg} that there exists a unique solution $u_{\lambda,w}\in H_0^1(\Omega)\cap L^\infty(\Omega)$ to \eqref{continuumproblem}. Hence, we are allowed to define the map
\[K:\re\times L^\infty(\Omega)\to L^\infty(\Omega),\quad (\lambda,w)\mapsto K(\lambda,w)=u_{\lambda,w}.\]

We will prove next that that $K$ is a completely continuous operator, i.e., it is continuous and maps bounded sets to relatively compact sets.

\begin{proposition}\label{Kcompact}
Assume that \eqref{H1} holds. Then, the operator $K$ is completely continuous.
\end{proposition}

\begin{proof}
We first prove that $K$ is continuous. Indeed, let $\{(\lambda_n,w_n)\}$ be a sequence in $\re\times L^\infty(\Omega)$ such that $(\lambda_n,w_n)\to (\lambda,w)$ for some $(\lambda,w)\in \re\times L^\infty(\Omega)$. Let us denote $u_n=K(\lambda_n,w_n)$, and let $B>0$ be such that $(\lambda_n^++1)w_n^+\leq B$. We know from Proposition \ref{exislambdaneg} that there exists $v\in H_0^1(\Omega)\cap L^\infty(\Omega)$ such that
\begin{equation*}
\begin{cases}
\dys-\Delta v+v=\mu(x)\frac{|\nabla v|^q}{v^\alpha}+f(x)+B\quad &\text{ in }\Omega,
\\
v>0 \quad &\text{ in }\Omega,
\\
v=0\quad &\text{ on }\partial\Omega.
\end{cases}
\end{equation*}
Hence, by virtue of Theorem \ref{comprinc1} (see also Remark \ref{lambda1remark1}), we deduce that 
\[u_n\leq v\leq \|v\|_{L^\infty(\Omega)}.\]
In particular, $\{u_n\}$ is bounded in $L^\infty(\Omega)$. 

\smallskip

Now we can argue as in \cite[Appendix]{CLLM} to prove that $\{u_n\}$ is, in fact, bounded in $C^{0,\eta}(\overline{\Omega})$ for some $\eta\in (0,1)$. Therefore, Arzel\`a-Ascoli theorem implies that $\{u_n\}$ admits a uniformly convergent subsequence. Say, up to a not relabeled subsequence, $u_n\to u$ uniformly in $\overline{\Omega}$ for some $u\in C(\overline{\Omega})$.

\smallskip

On the other hand, taking $u_n$ as test function in the weak formulation of \eqref{continuumproblem} yields
\[\int_\Omega|\nabla u_n|^2+\int_\Omega u_n^2=\int_\Omega\mu(x)|\nabla u_n|^q u_n^{1-\alpha}+\int_\Omega (f(x)+(\lambda_n^++1)w_n^+.\]
Using that $\{u_n\}$ and $\{(\lambda_n,w_n)\}$ are bounded in $L^\infty(\Omega)$ and in $\re\times L^\infty(\Omega)$, and also that $\alpha<q-1<1$, the previous equality clearly implies that $\{u_n\}$ is bounded in $H_0^1(\Omega)$. Then, $u\in H_0^1(\Omega)$ and, up to a new subsequence, $u_n\rightharpoonup u$ in $H_0^1(\Omega)$. Moreover, by \cite{BM}, $\nabla u_n\to \nabla u$ strongly in $L^q(\Omega)^N$. Furthermore, a lower local estimate on $\{u_n\}$ can be derived by comparison in the usual way. With all these estimates and convergences, the passing to the limit in \eqref{continuumproblem} is standard. 

\smallskip 

Therefore, $u\in H_0^1(\Omega)\cap L^\infty(\Omega)$ is the unique solution to \eqref{continuumproblem}. This means that $K(\lambda,w)=u$. Thus, we have proved that, up to a subsequence, $K(\lambda_n,w_n)\to K(\lambda,w)$ strongly in $L^\infty(\Omega)$. Actually, since $(\lambda,w)$ was fixed from the beginning, the whole sequence, and not just a subseqence, converges to $(\lambda,w)$. That is to say, $K$ is continuous.

\smallskip

It is left to prove that $K$ maps bounded sets to relatively compact sets. In other words, that for every sequence $\{(\lambda_n,w_n)\}$ bounded in $\re\times L^\infty(\Omega)$, there exists $(\lambda,w)\in \re\times L^\infty(\Omega)$ such that, up to a subsequence, $K(\lambda_n,w_n)\to K(\lambda,w)$ strongly in $L^\infty(\Omega)$. Indeed, it is well-known that, up to a subsequence, $\lambda_n\to\lambda$ in $\re$ and $w_n\to w$ weakly* in $L^\infty(\Omega)$ for some $(\lambda,w)\in \re\times L^\infty(\Omega)$. This convergence is enough to pass to the limit in the term with $w_n$. In the rest of the  terms, we pass to limit arguing as above. Thus, up to a subsequence, $K(\lambda_n,w_n)\to K(\lambda,w)$, and the proof is finished. 
\end{proof}

Let us define $\Phi(\lambda,u)=u-K(\lambda,u)$, and
\[\Sigma=\{(\lambda,u)\in\re\times L^\infty(\Omega):\Phi(\lambda,u)=0\}.\]
For any $\lambda_0\in\re$ and any isolated solution $u_0\in L^\infty(\Omega)$ to the equation $\Phi(\lambda_0,u)=0$, the Leray-Schauder degree $\deg(\Phi(\lambda_0,\cdot),B_r(u_0),0)$ is well defined and is constant for $r>0$ small enough. Thus it is possible to define the so called \emph{index} as
\[i(\Phi(\lambda_0,\cdot),u_0)=\lim_{r\to 0}\deg(\Phi(\lambda_0,\cdot),B_r(u_0),0).\]

\begin{proposition}\label{prop:continuum}
Assume that \eqref{H1} holds, and suppose also that $(P_0)$ has a solution $u_0$. Then, there exist two unbounded connected sets $\Sigma^-,\Sigma^+\subset\Sigma$ such that $\Sigma^-\subset(-\infty,0]\times L^\infty(\Omega)$, $\Sigma^+\subset~[0,\infty)\times L^\infty(\Omega)$ and $(0,u_0)\in\Sigma^-\cap\Sigma^+$. 
\end{proposition}

\begin{remark}\label{rem:continuum}
Observe that, if $\lambda\geq 0$, solving the equation $\Phi(\lambda,u)=0$ is equivalent to finding a solution to \eqref{problem}. In particular, the projection of $\Sigma^+$ onto $L^\infty(\Omega)$ is actually made of solutions to \eqref{problem}.
\end{remark}

\begin{proof}[Proof of Proposition \ref{prop:continuum}]
By virtue of Proposition \ref{Kcompact}, $K$ is completely continuous. Moreover, since $(P_0)$ admits at most one solution (by virtue of \cite{ACJT2}), then $u_0$ is the unique solution to $\Phi(0,u)=0$ (see Remark \ref{rem:continuum}). In particular, it is isolated. We will prove now that $i(\Phi(0,\cdot),u_0)\not =0$ by using the properties of the Leray-Schauder degree. 

\smallskip

Indeed, let $T:[0,1]\times L^\infty(\Omega)\to L^\infty(\Omega)$ be defined as $T(t,w)=u$, where $u\in H_0^1(\Omega)\cap L^\infty(\Omega)$ is the unique solution to the problem
\begin{equation*}
\begin{cases}
\dys-\Delta u+u=(1-t)\mu(x)\frac{|\nabla u|^q}{u^\alpha}+f(x)+w^+\quad &\text{ in }\Omega,
\\
u>0 \quad &\text{ in }\Omega,
\\
u=0\quad &\text{ on }\partial\Omega.
\end{cases}
\end{equation*}
It is easy to prove that $T$ is continuous and $T(t,\cdot):L^\infty(\Omega)\to L^\infty(\Omega)$ is completely continuous arguing as in the proof of Proposition \ref{Kcompact}. Moreover, for any $t\in [0,1]$, the unique solution $u_t\in H_0^1(\Omega)\cap L^\infty(\Omega)$ to $T(t,u_t)=u_t$ satisfies, thanks to Theorem \ref{comprinc1} (see also Remark \ref{lambda1remark1}), that $u_t\leq u_0\leq \|u_0\|_{L^\infty(\Omega)}$. Hence, if we set $\Psi_t(u)=u-T(t,u)$ and $R=2\|u_0\|_{L^\infty(\Omega)}$, we have that $\Psi_t(u)\not =0$ for every $t\in [0,1]$ and every $u\in\partial B_R(0)=\partial\{v\in L^\infty(\Omega):\|v\|_{L^\infty(\Omega)}<R\}$. Therefore, the homotopy property of the degree shows that
\[\deg(\Psi_0,B_R(0),0)=\deg(\Psi_1,B_R(0),0)\not=0.\]

On the other hand, let $r>0$ be small enough so that $B_r(u_0)\subset\subset B_R(0)$. Let us denote the following open, bounded and disjoint subsets of $B_R(0)$ as $A_1=B_r(u_0)$ and $A_2=B_R(0)\setminus \overline{B_r(u_0)}$. Since $u_0$ is unique, then $\Psi_0(u)\not=0$ for all $u\in \overline{B_R(0)}\setminus (A_1\cup A_2)=\partial B_R(0)\cup\partial B_r(u_0)$. Then, the additivity property of the degree implies that
\[\deg(\Psi_0,B_R(0),0)=\deg(\Psi_0,A_1,0)+\deg(\Psi_0,A_2,0).\] 
Now, again by the uniqueness of $u_0$, we have that $\Psi_0(u)\not=0$ for all $u\in A_2$. Thus the solution property of the degree says that $\deg(\Psi_0,A_2,0)=0$. That is to say,
\[\deg(\Psi_0,B_R(0),0)=\deg(\Psi_0,B_r(u_0),0).\] 
Putting all together, we have proved that 
\[i(\Phi(0,\cdot),u_0)=\deg(\Phi(0,\cdot),B_r(u_0),0)=\deg(\Psi_0,B_r(u_0),0)\not=0.\]

In conclusion, we can now apply \cite[Theorem 2.2]{ACJT1}, which is essentially \cite[Theorem 3.2]{Rabinowitz}, and the proof is finished.
\end{proof}

\end{document}